\documentclass[preprint,11pt,3p]{elsarticle}
\usepackage{amsmath,amsthm,amsfonts,amssymb,arydshln}
\usepackage{xcolor}
\usepackage{algorithm}
\usepackage{multirow}
\usepackage{mathdots}
\usepackage{multicol,graphicx}
\usepackage{mathrsfs}
\usepackage{booktabs}

\usepackage{comment}
\usepackage{subcaption}
\usepackage{placeins}
\usepackage{threeparttable}
\usepackage{lipsum}
\usepackage{graphicx}
\usepackage{algpseudocode}
\usepackage{tikz-cd}
\algnewcommand{\Input}{\item[\textbf{Input:}]}
\algnewcommand{\Output}{\item[\textbf{Output:}]}

\usepackage{mathrsfs}
\usepackage{sectsty}
\definecolor{astral}{RGB}{46,116,181}
\sectionfont{\color{astral}}
\newtheorem{theorem}{Theorem}[section]
\newtheorem{lemma}[theorem]{Lemma}

\newtheorem{definition}[theorem]{Definition}
\newtheorem{example}[theorem]{Example}
\newtheorem{remark}[theorem]{Remark}
\usepackage{hyperref}
\usepackage{scalerel}
\usepackage{stackengine,wasysym}

\usepackage{tikz}
\pgfdeclarelayer{bg} 
\pgfsetlayers{bg,main}

\usepackage{tikz,xcolor}
\usetikzlibrary{positioning}
\definecolor{lime}{HTML}{A6CE39}
\definecolor{lightblue}{rgb}{0.0, 0.0, 0.5}
\DeclareRobustCommand{\orcidicon}{%
	\begin{tikzpicture}
	\draw[lime, fill=lime] (0,0)
	circle [radius=0.16]
	node[white] {{\fontfamily{qag}\selectfont \tiny ID}};
	\draw[white, fill=white] (-0.0625,0.095)
	circle [radius=0.007];
	\end{tikzpicture}
	\hspace{-2mm}
}
\foreach \x in {A, ..., Z}{%
	\expandafter\xdef\csname orcid\x\endcsname{\noexpand\href{https://orcid.org/\csname orcidauthor\x\endcsname}{\noexpand\orcidicon}}
}

\begin{document}
\begin{frontmatter}
\title{
A Fast Algorithm for Quaternion Tensor Singular Value Decomposition under the T-product}

\author{Daochang Zhang$^a$, Yue Zhao$^{b,*}$, Jingqian Li$^c$}
\address{
$^{a}$ College of Mathematics, Changchun Normal University, Changchun 130032, P.R. China.
\\\textit{E-mail}: \texttt{daochangzhang@126.com}\\

$^{b}$ Department of Math \& Stat, Georgia State University, Atlanta, GA 30303, USA.
\\\textit{E-mail}: \texttt{yuezhao0303@163.com}\\

$^{c}$ School of Sciences, Tianjin University of Technology, Tianjin, P.R. China.
\\\textit{E-mail}: \texttt{JingqianLi85@163.com}
}
\begin{abstract}
Quaternion tensor singular value decomposition (T-SVD) under the T-product provides a framework for analyzing multidimensional quaternion data. However, the T-SVD construction of Pan and Ng (SIAM J. Matrix Anal. Appl., 2024) \cite{PanNg2024} contains a structural gap: an unrestricted quaternion matrix SVD need not preserve the paired Fourier block structure required for tensor factor recovery, potentially leading to loss of T-unitarity and inaccurate reconstruction. 

To address this issue, we develop a structure-preserving construction that transforms each paired quaternion block into a complex matrix and maps its complex SVD factors back to quaternion factors with the required structure. This construction corrects the proof of \cite[Theorem 4]{PanNg2024} and leads to a corrected version of \cite[Algorithm 2]{PanNg2024}, hereafter referred to as the ``corrected algorithm'', ensuring T-unitary tensor factors and an f-diagonal core for arbitrary third-order quaternion tensors. A numerical counterexample demonstrates the failure of \cite[Algorithm 2]{PanNg2024} and verifies the corrected algorithm at machine precision. 

Notably, we propose a fast structure-preserving algorithm,  Algorithm~\ref{alg:t-fast}, and compare it numerically with the corrected algorithm. The experimental results show that Algorithm~\ref{alg:t-fast} reduces the decomposition time while maintaining near-machine-precision accuracy. Furthermore, its application to color video encryption yields high reconstruction quality, favorable ciphertext statistics, and stability under input perturbations.
\end{abstract}


\end{frontmatter}

\section{Introduction}
Quaternion algebra provides a unified representation of multichannel data and has been widely used in signal processing, computer vision and facial recognition
~\cite{CHOU,SIAMMKB2013,KUNZE,MFBC2023,SHOE}. In particular, the three imaginary components of a quaternion can jointly encode the RGB channels of a color image. Quaternion tensors extend this representation to multidimensional data, allowing spatial and temporal structure to be retained without unfolding the data into matrices. These features motivate the development of quaternion tensor decompositions \cite{SIAMRevKB2009, LiuAMC2026, SIAMSMS2013, SIAMSL2018}.

The tensor singular value decomposition based on the T-product extends matrix SVD to third-order tensors through block-circulant representations~\cite{LAAKM2011}. Its extension to quaternion tensors is complicated by the noncommutativity of quaternion multiplication. Zheng et al.~\cite{ZMMAMLD} established a quaternion tensor SVD under the T-product for tensors symmetric in the third mode by using octonions. Pan and Ng~\cite{PanNg2024} developed a block diagonalization of quaternion circulant matrices and proposed a corresponding quaternion
tensor SVD and computational algorithm.

However, the proof of the quaternion tensor SVD in
\cite[Theorem~4]{PanNg2024} does not ensure that the computed singular-vector factors preserve the paired Fourier-block structure required for tensor recovery. In particular, the blocks associated with the paired frequencies $t$ and $p(t)$ satisfy specific algebraic relations between their complex and complementary quaternion
components. An unrestricted quaternion matrix SVD need not preserve these relations. Consequently, although the resulting factors form a valid matrix SVD, their inverse Fourier transforms need not be block-circulant and therefore cannot, in general, be identified with quaternion tensor factors. The same structural issue carries over to
the corresponding computational algorithm.

The remainder of this paper is organized as follows. Section~2 reviews the necessary preliminaries on quaternion tensors, the T-product, and the Fourier representation of quaternion block-circulant matrices, and establishes several auxiliary lemmas needed
in the subsequent analysis. Section~3 revisits the quaternion T-SVD. We identify the paired
Fourier-block structure required for tensor recovery, derive a corrected T-SVD that preserves this structure, and present a structure-preserving algorithm via the block-circulant matrix and a fast structure-preserving algorithm. Section~4 provides two numerical examples. The first gives a counterexample to the unrestricted construction and verifies the
corrected T-SVD through reconstruction and T-unitarity errors. The second compares Algorithm~\ref{alg:t-fast} and the corrected algorithm on random quaternion tensors in terms of computational time and reconstruction accuracy. Section~5 applies the proposed method to color video encryption and recovery and evaluates its reconstruction quality, ciphertext
statistics, differential sensitivity, and stability under input perturbations.

\section{Preliminaries}\label{sec:preliminaries}

Let $\mathbb R$, $\mathbb C$, and $\mathbb H$ denote the real field,
the complex field, and the quaternion algebra, respectively.
A quaternion is written as $q=q_0+q_1\mathbf i+q_2\mathbf j+q_3\mathbf k,~q_0,q_1,q_2,q_3\in\mathbb R,$
where $\mathbf i^2=\mathbf j^2=\mathbf k^2
=\mathbf i\mathbf j\mathbf k=-1.
$ Its conjugate and modulus are defined by $\overline q
=q_0-q_1\mathbf i-q_2\mathbf j-q_3\mathbf k,$ and $|q|=(q\overline q)^{1/2}.$ A quaternion is called pure if $q_0=0$ and unit if $|q|=1$. Orthogonality of pure quaternions refers to the Euclidean inner product of their imaginary coefficient vectors.
For a quaternion matrix $Z$, its conjugate transpose is denoted by
$Z^*=\overline Z^{\,T},$ where the conjugation is taken entrywise.

For the subsequent developments, let $\breve\mu$ and $\breve\alpha$
be orthogonal unit pure quaternions, and set $\breve\beta=\breve\mu\breve\alpha.$ Then $\{1,\breve\mu,\breve\alpha,\breve\beta\}$ forms an orthonormal basis of $\mathbb H$ over $\mathbb R$, and $$
\breve\mu^2=\breve\alpha^2=\breve\beta^2=-1,~~
\breve\mu\breve\alpha
=-\breve\alpha\breve\mu
=\breve\beta,~~
\breve\alpha\breve\beta
=-\breve\beta\breve\alpha
=\breve\mu,~~
\breve\beta\breve\mu
=-\breve\mu\breve\beta
=\breve\alpha.$$
The standard basis $\{1,\mathbf i,\mathbf j,\mathbf k\}$ is recovered by taking
$\breve\mu=\mathbf i,~
\breve\alpha=\mathbf j,~
\breve\beta=\mathbf k.$ Associated with $\breve\mu$, define the complex subfield $\mathbb C_{\breve\mu}
=\{x+y\breve\mu:x,y\in\mathbb R\},~
\mathbb C_{\breve\mu}\breve\alpha
=
\{z\breve\alpha:z\in\mathbb C_{\breve\mu}\},
$ then
$\mathbb H=\mathbb C_{\breve\mu}
\oplus
\mathbb C_{\breve\mu}\breve\alpha.$
Hence, every quaternion $q\in\mathbb H$ admits a unique representation $q=z+w\breve\alpha,~z,w\in\mathbb C_{\breve\mu}.$
For $z=x+y\breve\mu\in\mathbb C_{\breve\mu}$, its conjugate is
$\overline z=x-y\breve\mu.$  Moreover, $\breve\alpha z=\overline z\,\breve\alpha,~z\breve\alpha=\breve\alpha\,\overline z,~
z\in\mathbb C_{\breve\mu}.$ These relations follow from the quaternion basis representation and the commutation relations in~\cite{PanNg2024}.

Throughout the paper,
$\mathcal A=(a_{ijk})\in
\mathbb H^{n_1\times n_2\times n_3}$
denotes a third-order quaternion tensor. With respect to the basis
$\{1,\breve\mu,\breve\alpha,\breve\beta\}$, the Cayley--Dickson representation of $\mathcal A$ is written as $\mathcal A=\mathcal A_{\mathbf d}+\mathcal A_{\mathbf c}\breve\alpha,~\mathcal A_{\mathbf d},
\mathcal A_{\mathbf c}\in
\mathbb C_{\breve\mu}^{n_1\times n_2\times n_3}.$ We use MATLAB-style indexing: $\mathcal A(i,:,:)$, $\mathcal A(:,j,:)$, and $\mathcal A(:,:,k)$ denote its $i$-th horizontal, $j$-th lateral, and $k$-th frontal slices, respectively. For conciseness, the $k$-th frontal slice is denoted by $\mathcal A^{(k)}:=\mathcal A(:,:,k).
$ The Frobenius norm of $\mathcal A$ is defined by
$\|\mathcal A\|_F=\left(
\sum_{i=1}^{n_1}
\sum_{j=1}^{n_2}
\sum_{k=1}^{n_3}
|a_{ijk}|^2
\right)^{1/2}.$

The operator $\mathtt{unfold}$ stacks the frontal slices of tensor $\mathcal A$ vertically, and $\mathtt{fold}$ denotes its inverse. Thus, $\mathtt{fold}
\bigl(\mathtt{unfold}(\mathcal A)\bigr)=\mathcal A.$ The matrix $\operatorname{bcirc}(\mathcal A)\in\mathbb H^{am\times bm}$ is the block-circulant matrix whose first block column is
$\mathtt{unfold}(\mathcal A)$, with each subsequent block column obtained by a downward cyclic shift of the preceding one \cite{HeAMM}. Since its first block column uniquely determines $\mathcal A$, the mapping $\operatorname{bcirc}$ is injective. 

The quaternion discrete Fourier transform (QDFT) provides a convenient frequency-domain representation of quaternion block-circulant matrices. For a positive integer $m$, define the
normalized discrete quaternion Fourier matrix $\breve F_{\breve\mu}
\in\mathbb C_{\breve\mu}^{m\times m}$
by
$
(\breve F_{\breve\mu})_{rs}=
\frac{1}{\sqrt m}
\exp\left(
-\frac{2\pi\breve\mu}{m}(r-1)(s-1)
\right),~1\le r,~s\le m.
$
This is the equivalent $1$-based indexing form of the quaternion
Fourier matrix used in \cite{PanNg2024}. It is unitary and satisfies
$\breve F_{\breve\mu}^*
\breve F_{\breve\mu}
=
I_{n_3},~
\breve F_{\breve\mu}^{\,2}
=
P_{n_3},~
\breve\alpha\breve F_{\breve\mu}^*
=
\breve F_{\breve\mu}\breve\alpha$, where $P_{n_3}$ is defined by
\begin{align}\label{permutation}
P_{ij}=
\begin{cases}
1, & (i,j)=(1,1)\ \text{or}\ i+j=n_3+2,\quad 2\le i,j\le n_3,\\
0, & \text{otherwise}.
\end{cases}
\end{align}

We begin by recalling several basic definitions from
\cite{ZMMAMLD} that will be used throughout the paper.

\begin{definition}[T-product]\label{def:t-product}
Let
$\mathcal A\in\mathbb H^{n_1\times n_2\times n_3}$
and
$\mathcal B\in\mathbb H^{n_2\times n_4\times n_3}$.
The T-product of $\mathcal A$ and $\mathcal B$ is defined by
$\mathcal A\star\mathcal B
:=
\operatorname{fold}\big(
\operatorname{bcirc}(\mathcal A)
\operatorname{unfold}(\mathcal B)
\big)
\in\mathbb H^{n_1\times n_4\times n_3}.$
\end{definition}

\begin{definition}[Conjugate transpose]
\label{def:conjugate-transpose}
Let
$\mathcal A\in\mathbb H^{n_1\times n_2\times n_3}$.
The conjugate transpose of $\mathcal A$, denoted by $\mathcal A^*$,
is obtained by conjugate transposing each frontal slice of
$\mathcal A$ and then reversing the order of the resulting frontal
slices from $2$ to $n_3$.
\end{definition}

\begin{definition}[Identity tensor ]\label{def:identity-tensor}
The identity tensor
$\mathcal I\in\mathbb H^{n_1\times n_1\times n_3}$
has first frontal slice $I_{n_1}$ and all remaining frontal slices
equal to zero.
\end{definition}

\begin{definition}[T-unitary]\label{def:t-unitary}
A tensor
$\mathcal U\in\mathbb H^{n_1\times n_1\times n_3}$
is called T-unitary if
        $\mathcal U^*\star\mathcal U=\mathcal U\star\mathcal U^*
=\mathcal I.$
\end{definition}

\begin{definition}[F-diagonal tensor]\label{def:f-diagonal}
A tensor
$\mathcal S\in\mathbb H^{n_1\times n_2\times n_3}$
is called f-diagonal if every frontal slice
$\mathcal S^{(k)}$, $1\le k\le n_3$, is a diagonal matrix.
\end{definition}


\section{Quaternion T-SVD and Its Computational Algorithm}
\subsection{Quaternion T-SVD}
To establish the T-SVD, we first characterize the paired block structure that underlies the Fourier-domain construction. By combining
\cite[Theorem~3]{PanNg2024} with the Cayley--Dickson decomposition, we obtain the following result.
\begin{lemma}\label{thm:bcirc-fourier-equivalence}
Let $\mathcal X=\mathcal X_d+\mathcal X_c\breve\alpha
\in\mathbb H^{a\times b\times m}$ and $\mathcal X_d,\mathcal X_c
\in\mathbb C_{\breve\mu}^{a\times b\times m}$.
For $t=1,\ldots,m$, define the Fourier coefficients of $\mathcal X_d$ and $\mathcal X_c$ by
\small{\begin{align}\label{ddtt}
\breve D_{1,t}
=
\sum_{\ell=1}^{m}\exp\left(
-\frac{2\pi\breve\mu}{m}(t-1)(\ell-1)
\right)
\mathcal X_d^{(\ell)},\qquad
\breve D_{2,t}
=
\sum_{\ell=1}^{m}
\exp\left(
-\frac{2\pi\breve\mu}{m}(t-1)(\ell-1)
\right)
\mathcal X_c^{(\ell)}
\end{align}\small}
and $\breve D_{1,t},\breve D_{2,t}\in
\mathbb C_{\breve\mu}^{a\times b}$. Then the Fourier-domain representation associated with a block-circulant matrix of $\mathcal X$ is as follows
\begin{align}\label{equal}
(\breve F_{\breve\mu}\otimes I_a)
\operatorname{bcirc}(\mathcal X)
(\breve F_{\breve\mu}^{*}\otimes I_b)
=\breve D,\end{align} where
$\breve D=\operatorname{diag}
(\breve D_{1,1},\ldots,\breve D_{1,m})
+
\operatorname{diag}
(\breve D_{2,1},\ldots,\breve D_{2,m})
(P_m\otimes I_b)\breve\alpha.
$
\end{lemma}
\begin{proof}
Since  $
\mathcal X=\mathcal X_d+\mathcal X_c\breve\alpha$, and $
\mathcal X_d,\mathcal X_c\in\mathbb C_{\breve\mu}^{a\times b\times m}.$ Then
\[
\operatorname{bcirc}(\mathcal X)
=
\operatorname{bcirc}(\mathcal X_d)
+
\operatorname{bcirc}(\mathcal X_c)\breve\alpha.
\]
The Fourier diagonalization of block-circulant matrices and
the definitions in~\eqref{ddtt}, yields
\begin{align}\label{xd}
&(\breve F_{\breve\mu}\otimes I_a)
\operatorname{bcirc}(\mathcal X_d)
(\breve F_{\breve\mu}^{*}\otimes I_b)
=
\operatorname{diag}
(\breve D_{1,1},\ldots,\breve D_{1,m}),\\\notag
&(\breve F_{\breve\mu}\otimes I_a)
\operatorname{bcirc}(\mathcal X_c)
(\breve F_{\breve\mu}^{*}\otimes I_b)
=
\operatorname{diag}
(\breve D_{2,1},\ldots,\breve D_{2,m}).
\end{align}
Moreover,
\[\breve F_\mu^*\breve F_\mu=I_m,\quad
\breve F_\mu^2=P_m,\quad
\breve\alpha\breve F_{\breve\mu}^{*}
=
\breve F_{\breve\mu}\breve\alpha,
\quad
\breve F_{\breve\mu}
=
\breve F_{\breve\mu}^{*}P_m.
\]
It follows that
\[
\begin{aligned}
(\breve F_{\breve\mu}\otimes I_a)
\operatorname{bcirc}(\mathcal X_c)\breve\alpha
(\breve F_{\breve\mu}^{*}\otimes I_b)
&=
(\breve F_{\breve\mu}\otimes I_a)
\operatorname{bcirc}(\mathcal X_c)
(\breve F_{\breve\mu}^{*}\otimes I_b)
(P_m\otimes I_b)\breve\alpha
\\
&=
\operatorname{diag}
(\breve D_{2,1},\ldots,\breve D_{2,m})
(P_m\otimes I_b)\breve\alpha.
\end{aligned}
\]
Combining the above result with \eqref{xd} completes the proof. 
\end{proof}

\begin{remark}
The permutation matrix $P_m$ is defined in \eqref{permutation}. For each $t\ge2$, right multiplication by $P_m\otimes I_b$ moves the block
$\breve D_{2,t}\breve\alpha$ from position $(t,t)$
to position $(t,p)$, where
$p=m+2-t.$ Hence, for
$t=2,\ldots,\lfloor(m+1)/2\rfloor$ and $p=m+2-t$
the principal block submatrix of $\breve D$ indexed by $t$ and $p$ is
\begin{align}\label{paired structure}
\begin{bmatrix}
\breve D_{1,t} & \breve D_{2,t}\breve\alpha\\[0.2mm]
\breve D_{2,p}\breve\alpha & \breve D_{1,p}
\end{bmatrix}.
\end{align}
The index $t=1$ is self-paired; when $m$ is even,
$t=m/2+1$ is also self-paired.

Thus, any Fourier-domain representation associated with a block-circulant matrix of the tensor must possess this paired block structure \eqref{paired structure}, which must therefore be preserved when recovering tensor factors with the block-circulant representation through the inverse Fourier transform.

Here we follow the notation of \cite{PanNg2024}. We note that, in Theorem~4 of \cite{PanNg2024}, an unrestricted quaternion SVD of $\breve G_t$ guarantees the matrix decomposition and unitarity, but does not guarantee that the diagonal blocks of $\breve U_t$ and $\breve V_t$ belong to $\mathbb C_{\breve\mu}$ and that their off-diagonal blocks belong to $\mathbb C_{\breve\mu}\breve\alpha$. 

For example,
\[
I_2=
\operatorname{diag}(\breve\alpha,1)\,I_2\,
\operatorname{diag}(\breve\alpha,1)^*
\]
is a valid quaternion matrix SVD. Although $I_2$ satisfies the paired block structure 
\eqref{paired structure}, the displayed
unitary matrices do not: their first diagonal entry
$\breve\alpha$ lies outside $\mathbb C_{\breve\mu}$.

In other words, the paired block structure in \eqref{paired structure} is not necessarily preserved. Consequently, the resulting matrix factors cannot, in general, be transformed back through the inverse Fourier transform to tensor factors with block-circulant representations.
\end{remark}

To overcome the above difficulty, we introduce a complexification procedure that preserves the paired block structure \eqref{paired structure}. On this basis, we revise and complete the proof of the quaternion T-SVD and modify the corresponding algorithm. The resulting construction guarantees the validity of the decomposition in the tensor setting.
\begin{theorem}[T-SVD]
\label{cor:t-product-svd}
For every $\mathcal T\in\mathbb H^{n_1\times n_2\times m}$, there exist
$\mathcal U\in\mathbb H^{n_1\times n_1\times m}$ and
$\mathcal V\in\mathbb H^{n_2\times n_2\times m}$, unitary under the T-product,
and an f-diagonal tensor $\mathcal S\in\mathbb H^{n_1\times n_2\times m}$ such that
\begin{equation*}
\mathcal T=\mathcal U\star\mathcal S\star\mathcal V^{*}.
\end{equation*}
\end{theorem}

\begin{proof}
Let
$\breve D=
(\breve F_{\breve\mu}\otimes I_{n_1})
\operatorname{bcirc}(\mathcal T)
(\breve F_{\breve\mu}^{*}\otimes I_{n_2}).
$ By Lemma~\ref{thm:bcirc-fourier-equivalence}, $\breve D$ has the
paired block structure \eqref{paired structure}. At the fixed index $t=1$, and also at $t=m/2+1$ when $m$ is even, the corresponding diagonal block is
$
\breve D_{1,t}+\breve D_{2,t}\breve\alpha
\in\mathbb H^{n_1\times n_2},
$
which is an arbitrary quaternion matrix.

For the remaining indices, let $t=2,\ldots,\lfloor(m+1)/2\rfloor$ and $p=m+2-t$. The corresponding paired block submatrix of $\breve D$ is
\[
\breve G_t
=
\begin{bmatrix}
\breve D_{1,t} & \breve D_{2,t}\breve\alpha\\
\breve D_{2,p}\breve\alpha & \breve D_{1,p}
\end{bmatrix}
\in\mathbb H^{2n_1\times 2n_2}.
\]
For a positive integer $s$, define
\begin{align}\label{RL}
\breve L_s=\begin{bmatrix}I_s&0\\0&\breve\alpha I_s\end{bmatrix},
\qquad
\breve R_s=\breve L_s^*
=\begin{bmatrix}I_s&0\\0&-\breve\alpha I_s\end{bmatrix}.
\end{align}
Since $\breve{\alpha}^2=-1$, both $\breve L_s$ and $\breve R_s$ are unitary and $\breve R_s=\breve L_s^{-1}$. Using $\breve{\alpha}z=\bar{z}\breve{\alpha}$ for $z\in\mathbb{C}_{\breve{\mu}}$, we obtain
\begin{equation*}\label{eq:t-complex-matrix}
\breve H_t:=\breve L_{n_1}\breve G_t\breve R_{n_2}
=\begin{bmatrix}
\breve D_{1,t}&\breve D_{2,t}\\[1mm]
-\overline{\breve D_{2,p}}&\overline{\breve D_{1,p}}
\end{bmatrix}
\in\mathbb C_{\breve\mu}^{2n_1\times2n_2}.
\end{equation*} 
Take a complex SVD of $\breve H_t$ over $\mathbb C_{\breve{\mu}}$. After a suitable ordering of its singular triplets, we have
\begin{equation}\label{eq:t-complex-svd}
\breve H_t=U_t
\begin{bmatrix}\Sigma_t&0\\0&\Sigma_p\end{bmatrix}V_t^*,
\end{equation}
where $U_t\in\mathbb C_{\breve{\mu}}^{2n_1\times 2n_1}$ and $V_t\in\mathbb C_{\breve{\mu}}^{2n_2\times 2n_2}$ are unitary, while $\Sigma_t,\Sigma_p\in\mathbb R^{n_1\times n_2}$ are nonnegative rectangular diagonal matrices.
Now define
\begin{equation}\label{eq:t-unitary-matrices}
\widetilde U_t=\breve R_{n_1}U_t\breve L_{n_1},
\qquad
\widetilde V_t=\breve R_{n_2}V_t\breve L_{n_2}.
\end{equation}
The unitarity of $U_t$, $V_t$, $\breve L_s$, and $\breve R_s$ immediately implies that $\widetilde U_t$ and $\widetilde V_t$ are also unitary. Partition $U_t$ and $V_t$ conformably as
\begin{align*}
   U_t=\begin{bmatrix}
       U_{11}&U_{22}\\U_{21}&U_{22}
   \end{bmatrix},\qquad V_t=\begin{bmatrix}
       V_{11}&V_{22}\\V_{21}&V_{22}
   \end{bmatrix},
\end{align*}
where all the subblocks belong to $\mathbb C_{\breve\mu}$. By \eqref{RL}, it follows that
\begin{equation}\label{eq:t-unitary-entries}
\widetilde U_t=
\begin{bmatrix}
U_{11}&U_{12}\breve\alpha\\[1mm]
-\overline{U_{21}}\breve\alpha&\overline{U_{22}}
\end{bmatrix},\qquad
\widetilde V_t=
\begin{bmatrix}
V_{11}&V_{12}\breve\alpha\\[1mm]
-\overline{V_{21}}\breve\alpha&\overline{V_{22}}
\end{bmatrix}.
\end{equation}
Thus, the diagonal blocks belong to $\mathbb C_{\breve \mu}$, whereas the off-diagonal blocks belong to $\mathbb C_{\breve \mu}\breve \alpha$. Hence $\widetilde U_t$ and $\widetilde V_t$ have the same paired block structure as $\breve G_t$. Moreover, since $\Sigma_t$ and $\Sigma_p$ are real diagonal matrices,
they are invariant under the corresponding transformation. Consequently,
\[
\breve L_{n_1}
\begin{bmatrix}\Sigma_t&0\\0&\Sigma_p\end{bmatrix}
\breve R_{n_2}
=\begin{bmatrix}\Sigma_t&0\\0&\Sigma_p\end{bmatrix}.
\]
Combining \eqref{eq:t-complex-svd} and \eqref{eq:t-unitary-matrices} with the above invariance relation and the definition of $\breve H_t$ yields
\begin{equation*}\label{eq:t-block-decomposition}
\breve G_t=\widetilde U_t
\begin{bmatrix}\Sigma_t&0\\0&\Sigma_p\end{bmatrix}\widetilde V_t^*.
\end{equation*}
For the fixed index $t=1$, and also for $t=m/2+1$ when $m$ is even, we apply the quaternion matrix SVD to the corresponding diagonal block. Combining these factorizations with those of all paired blocks yields
\begin{align}\label{dusigmav}
\breve D
=
\breve U\Sigma\breve V^{*},
\qquad
\Sigma
=
\operatorname{diag}(\Sigma_1,\ldots,\Sigma_m),
\end{align}
where $\breve U$ and $\breve V$ are assembled from the factors
$\widetilde U_t$ and $\widetilde V_t$ associated with the paired indices, together with the corresponding quaternion SVD factors at the fixed indices.

Since the index pairs are disjoint, a simultaneous permutation that groups each pair \((t,p)\) transforms \(\breve U\) and \(\breve V\) into block-diagonal matrices whose diagonal blocks are unitary. It follows that both $\breve U$ and $\breve V$ are unitary.

Moreover, by \eqref{eq:t-unitary-entries}, each paired block of $\breve U$ and $\breve V$ has the paired block structure \eqref{paired structure}. At each fixed index, the corresponding quaternion block admits a unique decomposition into its $\mathbb C_{\breve\mu}$ and $\mathbb C_{\breve\mu}\breve\alpha$ components. Consequently, there exist matrices
\[
U_{1,t},U_{2,t}\in\mathbb C_{\breve\mu}^{n_1\times n_1},
\qquad
V_{1,t},V_{2,t}\in\mathbb C_{\breve\mu}^{n_2\times n_2},
\]
such that
\begin{align*}
\breve U
&=
\operatorname{diag}(U_{1,1},\ldots,U_{1,m})
+
\operatorname{diag}(U_{2,1},\ldots,U_{2,m})
(P_m\otimes I_{n_1})\breve\alpha,\\
\breve V
&=
\operatorname{diag}(V_{1,1},\ldots,V_{1,m})
+
\operatorname{diag}(V_{2,1},\ldots,V_{2,m})
(P_m\otimes I_{n_2})\breve\alpha.
\end{align*}
Let $\mathcal U_d$ and $\mathcal U_c$ be obtained by applying the inverse Fourier transform along the third mode to
$\{U_{1,t}\}_{t=1}^{m}$ and $\{U_{2,t}\}_{t=1}^{m}$, respectively, and set $\mathcal U=\mathcal U_d+\mathcal U_c\breve\alpha.
$ The tensor $\mathcal V$ is defined analogously. By Lemma~\ref{thm:bcirc-fourier-equivalence}, it follows that
\begin{equation}\label{reverseuv}
\begin{aligned}
\operatorname{bcirc}(\mathcal U)
&=
(\breve F_{\breve\mu}^{*}\otimes I_{n_1})
\breve U
(\breve F_{\breve\mu}\otimes I_{n_1}),~
\operatorname{bcirc}(\mathcal V^{*})=
(\breve F_{\breve\mu}^{*}\otimes I_{n_2})
\breve V^{*}
(\breve F_{\breve\mu}\otimes I_{n_2}).
\end{aligned}
\end{equation}
Similarly, let $\mathcal S$ be obtained by applying the inverse
$\breve\mu$-Fourier transform along the third mode to
$\{\Sigma_t\}_{t=1}^{m}$. Since each $\Sigma_t$ is a real diagonal
matrix, every frontal slice of $\mathcal S$ is diagonal. Hence
$\mathcal S$ is $f$-diagonal, and
\begin{align}\label{Ssigma}
\operatorname{bcirc}(\mathcal S)
=
(\breve F_{\breve\mu}^{*}\otimes I_{n_1})
\Sigma
(\breve F_{\breve\mu}\otimes I_{n_2}).
\end{align}
By \cite[Lemma 1]{m=ab} and Definition \ref{def:t-product}, it follows that
\begin{align}\label{xyxy}
\operatorname{bcirc}(\mathcal X\star\mathcal Y)
=
\operatorname{bcirc}(\mathcal X)
\operatorname{bcirc}(\mathcal Y).
\end{align}
Transforming \eqref{dusigmav} back to the tensor setting, and using
\eqref{reverseuv}, \eqref{Ssigma} and \eqref{xyxy}, we obtain
\[
\operatorname{bcirc}(\mathcal T)
=
\operatorname{bcirc}(\mathcal U)
\operatorname{bcirc}(\mathcal S)
\operatorname{bcirc}(\mathcal V^*)
=
\operatorname{bcirc}(\mathcal U\star\mathcal S\star\mathcal V^*).
\]
Since $\operatorname{bcirc}$ is injective, it follows that
$\mathcal T=\mathcal U\star\mathcal S\star\mathcal V^*.$
Moreover, since $\breve U$ and $\breve V$ are unitary, \eqref{reverseuv} shows that
$\operatorname{bcirc}(\mathcal U)$ and $\operatorname{bcirc}(\mathcal V)$ are unitary. By \eqref{xyxy},
\[
\mathcal U^*\star\mathcal U
=
\mathcal U\star\mathcal U^*
=
\mathcal I,
\qquad
\mathcal V^*\star\mathcal V
=
\mathcal V\star\mathcal V^*
=
\mathcal I.
\]
Thus, $\mathcal U$ and $\mathcal V$ are $T$-unitary.
\end{proof}

\subsection{Computational Algorithm}
The constructive proof of Theorem~\ref{cor:t-product-svd} suggests two implementations of the quaternion T-SVD. The first is based on the block-circulant representation and preserves the paired block structure derived in the preceding subsection.
To avoid the cost of explicitly forming and manipulating the block-circulant matrix, we then develop an equivalent FFT-based implementation in the Fourier domain. The latter retains the same
decomposition while reducing the computational cost.

Throughout this subsection, we set $m=n_3$, $\breve\mu=\mathbf i$, and $\breve\alpha=\mathbf j$.
The mode-$3$ Fourier transform and its inverse are implemented by \texttt{fftq} and \texttt{ifftq}, respectively. We use $\operatorname{qsvd}_{\mathbb C}$ to denote the fast universal algorithm for QSVD in \cite[Algorithm~1]{DCTAML2026}.
We first present an implementation of Theorem~\ref{cor:t-product-svd}
using the block circulant representation.
\begin{algorithm}[H]
\caption{Structure-preserving algorithm for quaternion T-SVD via block-circulant matrices}
\label{alg:t-explicit}
\small
\begin{algorithmic}[1]
\Input $\mathcal A\in\mathbb H^{n_1\times n_2\times m}$.
\Output $\mathcal A=\mathcal U\star\mathcal S\star\mathcal V^*$,
with T-unitary $\mathcal U,\mathcal V$ and an f-diagonal
tensor $\mathcal S$.

\State Compute
$\breve D
=(\breve F_{\breve\mu}\otimes I_{n_1})
\mathtt{bcirc}(\mathcal A)
(\breve F_{\breve\mu}^*\otimes I_{n_2}).
$ Write $\breve D_{rs}\in\mathbb H^{n_1\times n_2}$
for its $(r,s)$ block.

\State Initialize
$\breve U=0\in\mathbb H^{mn_1\times mn_1}$,
$\Sigma=0\in\mathbb H^{mn_1\times mn_2}$, and
$\breve V=0\in\mathbb H^{mn_2\times mn_2}$.

\State For $t=1$, and also $t=m/2+1$ if $m$ is even,
compute
{\footnotesize
\[
(\breve U_t,\Sigma_t,\breve V_t)
=\operatorname{qsvd}_{\mathbb C}(\breve D_{tt}),
\]
}
and place these matrices in the respective $(t,t)$ blocks
of $\breve U,\Sigma,\breve V$.

\For{$t=2,\ldots,\lfloor(m+1)/2\rfloor$}
    \State Set $p=m+2-t$ and form
    {\footnotesize
    \[
    \breve G_t=
    \begin{bmatrix}
    \breve D_{tt}&\breve D_{tp}\\
    \breve D_{pt}&\breve D_{pp}
    \end{bmatrix},
    \qquad
    \breve H_t
    =\breve L_{n_1}\breve G_t\breve R_{n_2}
    \in\mathbb C^{2n_1\times2n_2}.
    \]
    }

    \State Compute the complex SVD with the column order described above:
    $
    \breve H_t
    =U_t\operatorname{diag}(\Sigma_t,\Sigma_p)V_t^*.
    $

    \State Set
    $\breve U_t=\breve R_{n_1}U_t\breve L_{n_1}$ and
    $\breve V_t=\breve R_{n_2}V_t\breve L_{n_2}$.

    \State Insert the four blocks of $\breve U_t$ and
    $\breve V_t$, in their existing order, into positions
    $(t,t),(t,p),(p,t),(p,p)$ of $\breve U$ and $\breve V$.

    \State Set
    $\Sigma_{tt}=\Sigma_t$ and $\Sigma_{pp}=\Sigma_p$.
\EndFor

\State Compute
{\footnotesize
\begin{align*}
\mathtt{bcirc}(\mathcal U)
=(\breve F_{\breve\mu}^*\otimes I_{n_1})
  \breve U(\breve F_{\breve\mu}\otimes I_{n_1}),~
\mathtt{bcirc}(\mathcal S)
=(\breve F_{\breve\mu}^*\otimes I_{n_1})
  \Sigma(\breve F_{\breve\mu}\otimes I_{n_2}),~
\mathtt{bcirc}(\mathcal V)
=(\breve F_{\breve\mu}^*\otimes I_{n_2})
  \breve V(\breve F_{\breve\mu}\otimes I_{n_2}).
\end{align*}
}
\State Recover $\mathcal U,\mathcal S,\mathcal V$
by applying $\mathtt{fold}$ to the first block column
of each corresponding block circulant matrix.

\end{algorithmic}
\end{algorithm}
The preceding algorithm requires the explicit construction of the block-circulant matrix. To reduce the associated computational cost, we
exploit its Fourier-domain representation and compute the same T-SVD directly in the Fourier domain. This yields the following algorithm.
\begin{algorithm}[H]
\caption{Fast structure-preserving algorithm for quaternion T-SVD via FFTs}
\label{alg:t-fast}
\small
\begin{algorithmic}[1]
\Input $\mathcal A\in\mathbb H^{n_1\times n_2\times m}$.
\Output $\mathcal A=\mathcal U\star\mathcal S\star\mathcal V^*$,
with T-unitary $\mathcal U,\mathcal V$ and an f-diagonal tensor $\mathcal S$.

\State Compute the complex tensors
$
\breve{\mathcal D}_1
\gets\operatorname{fft}(\mathcal A_{\mathbf d},[],3),$ and $
\breve{\mathcal D}_2
\gets\operatorname{fft}(\overline{\mathcal A_{\mathbf c}},[],3).
$ Write $\breve D_{1,t}=\breve{\mathcal D}_1^{(t)}$ and $\breve D_{2,t}=\breve{\mathcal D}_2^{(t)}$
for their $t$-th frontal slices.

\State Initialize
$\widehat{\mathcal U}\in\mathbb H^{n_1\times n_1\times m}$,
$\widehat{\mathcal S}\in\mathbb H^{n_1\times n_2\times m}$, and
$\widehat{\mathcal V}\in\mathbb H^{n_2\times n_2\times m}$
with zeros.

\State For $t=1$, and also $t=m/2+1$ if $m$ is even, set
{\footnotesize
\[
(\widehat{\mathcal U}^{(t)},
 \widehat{\mathcal S}^{(t)},
 \widehat{\mathcal V}^{(t)})
=\operatorname{qsvd}_{\mathbb C}
(\breve D_{1,t}+\breve D_{2,t}\mathbf j).
\]
}
\For{$t=2,\ldots,\lfloor(m+1)/2\rfloor$}
    \State Set $p=m+2-t$ and form the complex matrix
    {\footnotesize
    \[
    \breve H_t=
    \begin{bmatrix}
    \breve D_{1,t}&\breve D_{2,t}\\[1mm]
    -\overline{\breve D_{2,p}}&\overline{\breve D_{1,p}}
    \end{bmatrix}.
    \]
    }
    \State Compute the complex SVD with the column order
    described above:
    $
    \breve H_t
    =U_t\operatorname{diag}(\Sigma_t,\Sigma_p)V_t^*.
    $

    \State Set and partition
    {\footnotesize
    \[
    \breve U_t=\breve R_{n_1}U_t\breve L_{n_1}
    =
    \begin{bmatrix}
    \breve U_{11}&\breve U_{12}\\
    \breve U_{21}&\breve U_{22}
    \end{bmatrix},
    \qquad
    \breve V_t=\breve R_{n_2}V_t\breve L_{n_2}
    =
    \begin{bmatrix}
    \breve V_{11}&\breve V_{12}\\
    \breve V_{21}&\breve V_{22}
    \end{bmatrix}.
    \]
    }

    \State Store the Fourier slices:
    {\footnotesize
    \begin{align*}
    &\widehat{\mathcal U}^{(t)}
    =\breve U_{11}+\breve U_{12},~
    \widehat{\mathcal V}^{(t)}
    =\breve V_{11}+\breve V_{12},~
    \widehat{\mathcal S}^{(t)}=\Sigma_t,\\&
    \widehat{\mathcal U}^{(p)}
    =\breve U_{21}+\breve U_{22},~
    \widehat{\mathcal V}^{(p)}
    =\breve V_{21}+\breve V_{22},~
    \widehat{\mathcal S}^{(p)}=\Sigma_p.
    \end{align*}}
\EndFor

\State
$\mathcal U\gets\mathtt{ifftq}(\widehat{\mathcal U},[],3)$,
$\mathcal S\gets\mathtt{ifftq}(\widehat{\mathcal S},[],3)$,
$\mathcal V\gets\mathtt{ifftq}(\widehat{\mathcal V},[],3)$.

\end{algorithmic}
\end{algorithm}
In Algorithm~\ref{alg:t-fast}, $\breve D_{1,t}$ and
$\breve D_{2,t}$ denote the $t$-th frontal slices of the two FFT outputs, respectively. The matrix $\breve H_t$ is precisely have the paired block structure \eqref{paired structure}, and the sums in the storage step combine the two nonzero blocks in each corresponding block row according to \eqref{equal}. Hence, the quaternion T-SVD can be computed using FFTs and small matrix SVDs without explicitly forming the block-circulant matrices or Kronecker products. Algorithm~\ref{alg:t-fast} therefore provides a corrected and accelerated version of the quaternion T-SVD algorithm in \cite[Algorithm~2]{PanNg2024}, preserving the required paired block structure while avoiding the explicit construction of the full block-circulant matrix.

\section{Numerical examples}
\subsection{Example 1}
This example illustrates the structural deficiency of the unrestricted quaternion SVD in the passage from the matrix decomposition to the tensor decomposition, and verifies the corrected construction proposed above. All computations were performed in MATLAB R2022b with QTFM 3.7, and the numerical values are reported to four decimal places.
\begin{example}
    For the numerical implementation, we take $\breve\mu=\mathbf i$, $\breve\alpha=\mathbf j$, and $\breve\beta=\mathbf k$.
Generate $\mathcal A\in\mathbb H^{2\times2\times4}$ by
\texttt{randq(2,2,4)} with the Twister random generator and seed $20260917$.
The Fourier transform of $\operatorname{bcirc}(\mathcal A)$,
restricted to block rows and columns $2$ and $4$, gives
{\small\setlength{\arraycolsep}{3pt}
\[
\breve G_2\approx
\left[\begin{array}{cc|cc}
-1.0692+1.2258\mathbf i & 0.5285-0.5854\mathbf i & 1.5490\mathbf j-1.0009\mathbf k & 0.6640\mathbf j-0.2094\mathbf k \\
0.4284-0.2049\mathbf i & 1.3895-2.7648\mathbf i & 0.6109\mathbf j-0.6712\mathbf k & 1.6522\mathbf j-0.3722\mathbf k \\
\hline
0.1031\mathbf j+1.6527\mathbf k & -1.1563\mathbf j-0.0844\mathbf k & -0.5221-0.7450\mathbf i & 3.1291+0.6682\mathbf i \\
-1.5335\mathbf j-2.3936\mathbf k & -0.8445\mathbf j-0.1057\mathbf k & -0.3083-0.8920\mathbf i & 0.2070-0.6972\mathbf i
\end{array}\right].
\]
}
The four singular values of $\breve G_2$ are $4.895044$, $3.888698$, $2.577465$, and $0.543545$. They are positive and distinct, and
$\operatorname{bcirc}(\mathcal A)$ is also nonsingular.

We first apply the unrestricted quaternion SVD used in the implementation of
\cite{PanNg2024},
$$\breve G_2=\breve U_2^{\mathrm{old}}\Sigma
(\breve V_2^{\mathrm{old}})^*.$$ 
The resulting matrices $\breve U_2^{\mathrm{old}}$ and $\breve V_2^{\mathrm{old}}$ are shown below.
{\small\setlength{\arraycolsep}{3pt}
\[
\breve U_2^{\mathrm{old}}\approx
\left[\begin{array}{cc|cc}
-0.1248+0.3478\mathbf i & 0.3407-0.1594\mathbf i & 0.2628+0.4145\mathbf i & -0.6031+0.3425\mathbf i \\
-0.0447+0.6178\mathbf i & -0.0861+0.3847\mathbf i & 0.5217-0.1162\mathbf i & 0.1382-0.3951\mathbf i \\
\hline
0.3195\mathbf j+0.5817\mathbf k & 0.2047\mathbf j-0.2720\mathbf k & 0.5352\mathbf j-0.3754\mathbf k & 0.0750\mathbf j+0.1028\mathbf k \\
0.1018\mathbf j-0.1703\mathbf k & 0.4891\mathbf j+0.5898\mathbf k & -0.0445\mathbf j-0.2098\mathbf k & 0.4040\mathbf j+0.4053\mathbf k
\end{array}\right],
\]
}
{\small\setlength{\arraycolsep}{3pt}
\[
\breve V_2^{\mathrm{old}}\approx
\left[\begin{array}{cc|cc}
0.3391 & -0.8398 & 0.1860 & -0.3809 \\
-0.5161+0.2730\mathbf i & -0.4112+0.1354\mathbf i & 0.1611+0.3888\mathbf i & 0.5258+0.1345\mathbf i \\
\hline
0.1029\mathbf j+0.1713\mathbf k & -0.2463\mathbf j+0.2109\mathbf k & -0.0729\mathbf j+0.6999\mathbf k & 0.5990\mathbf j+0.0293\mathbf k \\
0.4061\mathbf j+0.5825\mathbf k & 0.0449\mathbf j+0.0114\mathbf k & 0.2204\mathbf j-0.4945\mathbf k & 0.3703\mathbf j+0.2520\mathbf k
\end{array}\right].
\]
}
In both matrices, the two blocks in the upper block row have entries in $\mathbb C_{\mathbf i}$, whereas the two blocks in the lower block row
have entries in $\mathbb C_{\mathbf i}\mathbf j$. Hence, although the quaternion SVD gives a valid matrix factorization and unitary matrix
factors, the resulting $\breve U_2^{\mathrm{old}}$ and $\breve V_2^{\mathrm{old}}$ do not have the required paired block structure \eqref{paired structure}. Consequently, the inverse Fourier transforms of these factors are not block-circulant matrices and therefore cannot be identified with \(\operatorname{bcirc}(\mathcal U)\) and \(\operatorname{bcirc}(\mathcal V)\) for tensor factors \(\mathcal U\) and \(\mathcal V\).

We now apply the corrected construction. Define
\[
L_2=\operatorname{diag}(I_2,\mathbf jI_2),\qquad
R_2=L_2^*,\qquad H_2=L_2\breve G_2R_2\in\mathbb C_{\mathbf i}^{4\times4}.
\]
Thus, the quaternion paired block is converted into a complex matrix over $\mathbb C_{\mathbf i}$.
Compute the complex SVD, $H_2=\widehat U_2\Sigma\widehat V_2^*$ and transform the singular-vector factors back by
\[
\breve U_2^{\mathrm{new}}=R_2\widehat U_2L_2,\qquad
\breve V_2^{\mathrm{new}}=R_2\widehat V_2L_2.
\]
The resulting matrices are
{\small\setlength{\arraycolsep}{3pt}
\[
\breve U_2^{\mathrm{new}}\approx
\left[\begin{array}{cc|cc}
-0.1248+0.3478\mathbf i & 0.3407-0.1594\mathbf i & 0.2628\mathbf j+0.4145\mathbf k & -0.6031\mathbf j+0.3425\mathbf k \\
-0.0447+0.6178\mathbf i & -0.0861+0.3847\mathbf i & 0.5217\mathbf j-0.1162\mathbf k & 0.1382\mathbf j-0.3951\mathbf k \\
\hline
0.3195\mathbf j+0.5817\mathbf k & 0.2047\mathbf j-0.2720\mathbf k & -0.5352+0.3754\mathbf i & -0.0750-0.1028\mathbf i \\
0.1018\mathbf j-0.1703\mathbf k & 0.4891\mathbf j+0.5898\mathbf k & 0.0445+0.2098\mathbf i & -0.4040-0.4053\mathbf i
\end{array}\right],
\]
}
{\small\setlength{\arraycolsep}{3pt}
\[
\breve V_2^{\mathrm{new}}\approx
\left[\begin{array}{cc|cc}
0.3391 & -0.8398 & 0.1860\mathbf j & -0.3809\mathbf j \\
-0.5161+0.2730\mathbf i & -0.4112+0.1354\mathbf i & 0.1611\mathbf j+0.3888\mathbf k & 0.5258\mathbf j+0.1345\mathbf k \\
\hline
0.1029\mathbf j+0.1713\mathbf k & -0.2463\mathbf j+0.2109\mathbf k & 0.0729-0.6999\mathbf i & -0.5990-0.0293\mathbf i \\
0.4061\mathbf j+0.5825\mathbf k & 0.0449\mathbf j+0.0114\mathbf k & -0.2204+0.4945\mathbf i & -0.3703-0.2520\mathbf i
\end{array}\right].
\]
}
Now the diagonal blocks have entries in $\mathbb C_{\mathbf i}$, while the
off-diagonal blocks have entries in $\mathbb C_{\mathbf i} \mathbf j$. Hence, \(\breve U_2^{\mathrm{new}}\) and \(\breve V_2^{\mathrm{new}}\) have exactly the paired block structure required \eqref{paired structure} for recovering the corresponding tensor factors.

To compare the two constructions quantitatively, define the relative matrix reconstruction error by
\[
e_{\mathrm{mat}}=
\frac{\|\breve G_2-\breve U_2\Sigma\breve V_2^*\|_F}
     {\|\breve G_2\|_F}.
\]

After incorporating the SVD factors at frequencies \(1\) and \(3\), recover the tensors \(\mathcal U\), \(\mathcal S\), and \(\mathcal V\) for each construction and evaluate the corresponding T-product reconstruction. Define
\[
e_{\mathrm{rec}}=
\frac{\|\mathcal A-\mathcal U*\mathcal S*\mathcal V^*\|_F}{\|\mathcal A\|_F},
\qquad
e_U=\frac{\|\mathcal U^* *\mathcal U-\mathcal I\|_F}{\|\mathcal I\|_F},
\qquad
e_V=\frac{\|\mathcal V^* *\mathcal V-\mathcal I\|_F}{\|\mathcal I\|_F}.
\]
The resulting matrix and tensor errors are
\begin{center}
\begin{tabular}{lcccc}
\hline
Construction & $e_{\mathrm{mat}}$ & $e_{\mathrm{rec}}$ & $e_U$ & $e_V$\\
\hline
Original & $3.3065\times10^{-15}$ & $1.2166$ & $7.0711\times10^{-1}$ & $7.0711\times10^{-1}$\\
Corrected & $3.2830\times10^{-15}$ & $2.8014\times10^{-15}$ & $8.2975\times10^{-16}$ & $6.1458\times10^{-16}$\\
\hline
\end{tabular}
\end{center}
Both constructions therefore reproduce the matrix factorization to machine precision. The distinction appears when the matrix factors are mapped back to the tensor setting. For the unrestricted quaternion SVD, the T-product reconstruction error is of order one, and the recovered factors fail the T-unitarity conditions. In contrast, the corrected construction yields both tensor reconstruction and T-unitarity errors at machine precision. This example therefore confirms that matrix-level SVD accuracy alone is insufficient; preservation of the paired block structure is essential for recovering a valid tensor SVD.
\end{example}

Algorithm~\ref{alg:t-fast} employs an implementation of
\cite[Algorithm~1]{DCTAML2026} to compute
$\operatorname{qsvd}_{\mathbb C}$.
For comparison, we replace $\operatorname{qsvd}_{\mathbb C}$ in Algorithm~\ref{alg:t-fast} with the quaternion \texttt{svd} routine in QTFM and refer to the resulting method as the corrected version of the T-SVD algorithm in \cite[Algorithm~2]{PanNg2024}. This method is labelled ``Corrected T-SVD'' in the figures.

\subsection{Example 2}
\subsubsection{Experimental Setup}
The input tensors are generated as
$
\mathcal A=\mathcal A_0+\mathcal A_1\mathbf i
             +\mathcal A_2\mathbf j+\mathcal A_3\mathbf k
\in\mathbb H^{n\times n\times16},
$
where the entries of the four real tensors are independent standard
normal random variables generated by MATLAB \texttt{randn}. We fix
$\breve\mu=\mathbf i$, $\breve\alpha=\mathbf j$, and
$\breve\beta=\mathbf k$, and select 16 values of $n$ using
\texttt{round(linspace(4,512,16))}. The random seed is 20260921.
One tensor is generated for each size and used by both implementations.

All computations are performed in MATLAB R2022b with QTFM 3.7 on an
Intel Core i5-13500HX computer with 16~GB of RAM, using one computational
thread. After an initial warm-up on the smallest input, each implementation
is run twice on the same tensor; the execution order is reversed in the
second run. The reported errors and times are the arithmetic means of
the two runs.

The relative reconstruction error is measured by
\begin{equation}\label{eq:tsvd-benchmark-error}
e_{\mathrm{rec}}
=\frac{\|\mathcal A-\mathcal U\star\mathcal S\star\mathcal V^*\|_F}
       {\|\mathcal A\|_F}.
\end{equation}
Thus the error is evaluated from the final tensor reconstruction under
the T-product. Decomposition time is measured by MATLAB
\texttt{tic}/\texttt{toc} and includes the forward FFTs, matrix SVDs,
and inverse FFTs. Input generation, reconstruction-error evaluation,
and file output are excluded from the timing.
\subsection{Results and discussion}

Figure~\ref{fig:tsvd-benchmark-time} shows that
Algorithm~\ref{alg:t-fast} is faster for all tested sizes, with speedup
factors ranging from $1.48$ to $5.13$ relative to the QTFM-based
implementation. For example, the mean times are $1.1988$~s and
$2.6424$~s at $n=241$, and $9.2497$~s and $47.4476$~s at $n=512$,
respectively. Since the two implementations differ only in the matrix
SVD routine used in Step~3, the observed timing difference reflects the
effect of this implementation.

Figure~\ref{fig:tsvd-benchmark-error} shows the relative reconstruction
errors on a logarithmic scale. For Algorithm~\ref{alg:t-fast}, the
errors range from $3.7430\times10^{-15}$ to
$3.5974\times10^{-14}$, compared with
$2.3296\times10^{-15}$ to $6.6910\times10^{-15}$ for the QTFM-based
implementation. At $n=512$, the corresponding errors are
$2.7778\times10^{-14}$ and $6.6910\times10^{-15}$. Although the
QTFM-based implementation yields slightly smaller reconstruction
errors, both methods achieve near-machine-precision accuracy.
Algorithm~\ref{alg:t-fast} therefore reduces the decomposition time
while maintaining comparable numerical accuracy over the tested sizes.

\begin{figure}[H]
\centering
\begin{subfigure}{0.49\linewidth}
\centering
\includegraphics[width=\linewidth]
{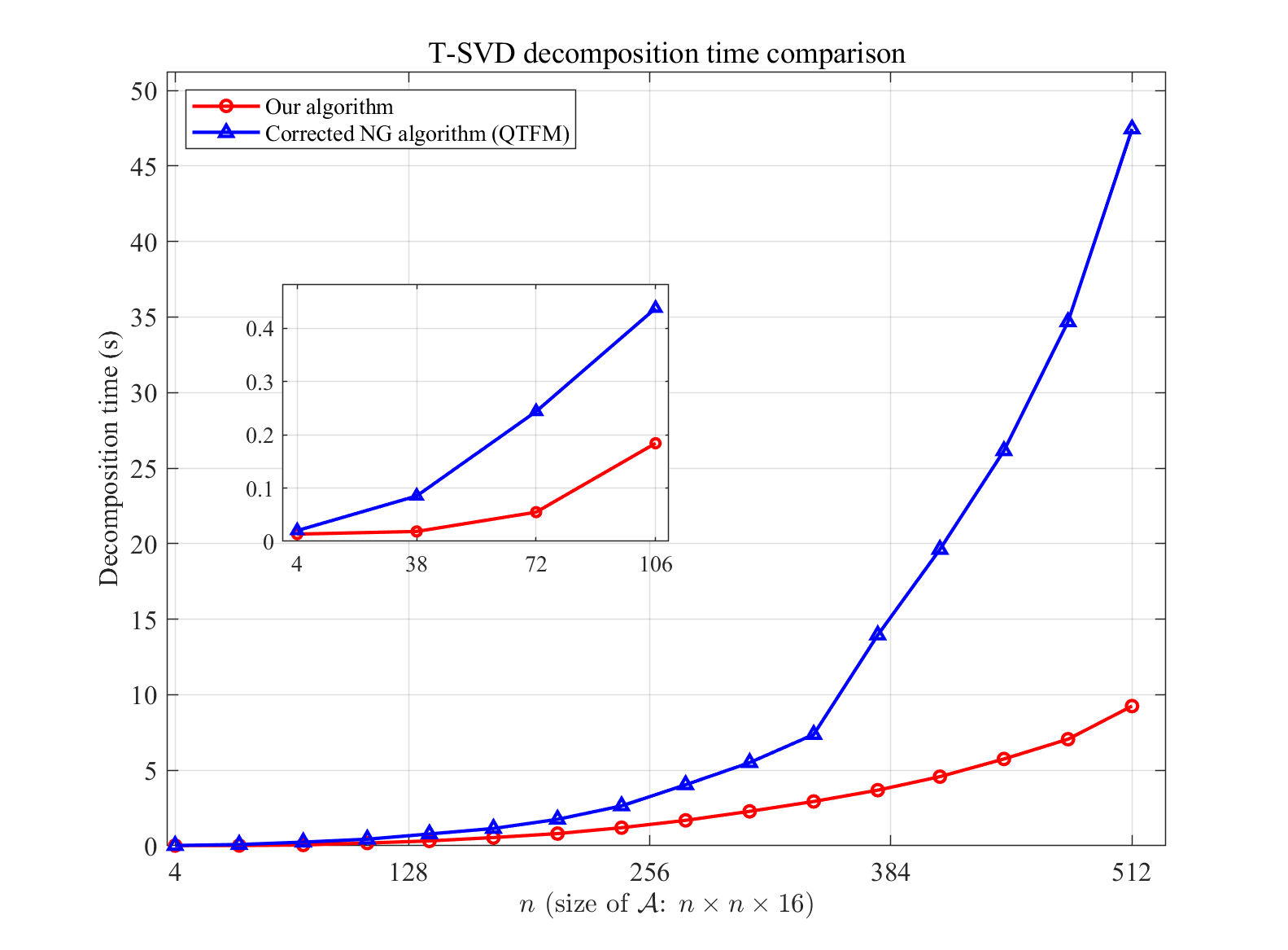}
\caption{Decomposition time.}
\label{fig:tsvd-benchmark-time}
\end{subfigure}
\hfill
\begin{subfigure}{0.49\linewidth}
\centering
\includegraphics[width=\linewidth]
{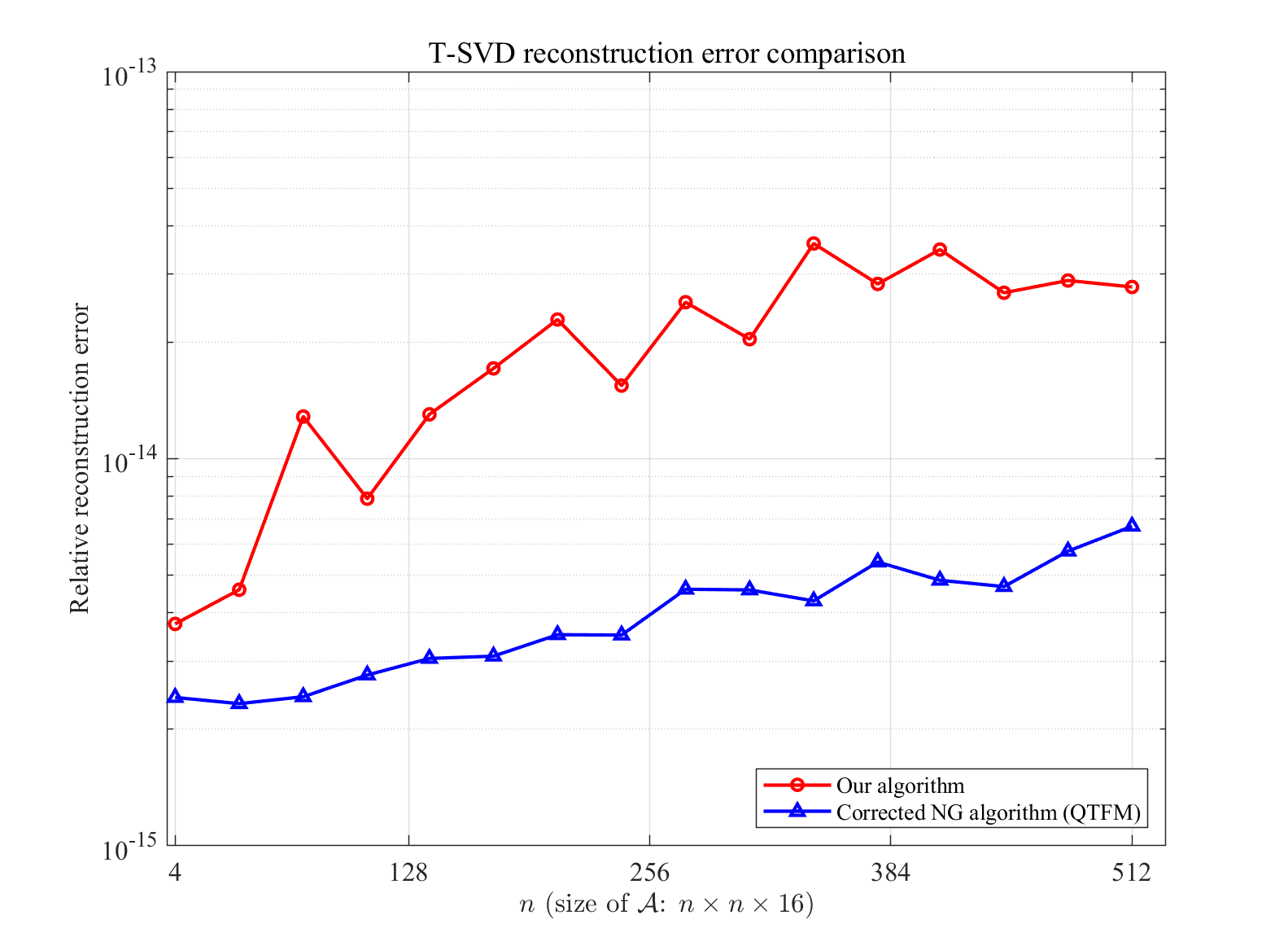}
\caption{Relative reconstruction error.}
\label{fig:tsvd-benchmark-error}
\end{subfigure}

\caption{Comparison of the two implementations for random quaternion
tensors of size $n\times n\times16$. Each point is the mean of two runs
on the same input. The time plot includes an enlargement of the first
four sizes.}
\label{fig:tsvd-benchmark}
\end{figure}

\section{Applications to Color Video Processing}
We apply Algorithm~\ref{alg:t-fast} to color video encryption and recovery using the \textit{kite-surf} and \textit{boat} sequences from the DAVIS 2017 TrainVal 480p dataset~\cite{DAVIS2017}. The two sequences contain 50 and 75 frames, respectively. Each frame is read from the JPEG sequence, converted to a double-precision RGB array with values in $[0,1]$, and resized to $256\times256$. The evaluation considers reconstruction quality, ciphertext statistics, and stability to input perturbations. All experiments are performed in MATLAB R2022b using the standard quaternion basis $\{\mathbf i,\mathbf j,\mathbf k\}$ and random seed 20260716.

\subsection{Encryption, Recovery, and Evaluation Metrics}
Let $\mathcal{A}=R\mathbf{i}+G\mathbf{j}+B\mathbf{k}$ denote the
quaternion tensor representing the RGB data. The decomposition and
subsequent reconstruction are given by
\[
\mathcal{A}
=\mathcal{U}\star\mathcal{S}\star\mathcal{V}^{*},
\qquad
\widetilde{\mathcal{A}}
=\mathcal{U}\star\widetilde{\mathcal{S}}\star\mathcal{V}^{*},
\]
where all singular values are retained and
$\widetilde{\mathcal{S}}$ is obtained by decrypting and rescaling
the encoded components of $\mathcal{S}$. The factors $\mathcal{U}$
and $\mathcal{V}$ are retained for reconstruction.
For each frontal slice of $\mathcal{S}$, the three imaginary
components are processed independently. Each component is min--max
scaled to $[0,255]$ and quantized to an 8-bit unsigned integer array,
denoted by $\overline{P}$. If the component range is smaller than
machine epsilon, its encoded array is set to zero. The componentwise
minima and maxima are retained for rescaling.

Encryption combines a logistic-map-generated XOR mask with an
Arnold permutation:
\begin{equation*}
y_{t+1}=3.97y_t(1-y_t),\quad
K_t=\operatorname{mod}(\lfloor10^{15}y_t\rfloor,256),\quad
C=\Pi^{20}(\overline P\oplus K),
\end{equation*}
Here, $\oplus$ denotes bitwise XOR, $K$ is the mask obtained by
reshaping the sequence $\{K_t\}$ to the dimensions of $\overline{P}$,
and $\Pi$ is the Arnold permutation with parameters $a=2$, $b=1$ and 20 permutation iterations \cite{Logistic,Arnold}. The initial logistic-map state is selected separately for
each slice and imaginary component.

Decryption applies the inverse permutation followed by the same
XOR mask, $\overline{P}=\Pi^{-20}(C)\oplus K,$
and rescales the recovered array using the stored extrema. The resulting imaginary components are combined with the retained real component to form $\widetilde{\mathcal{S}}$. Although no singular
values are discarded, reconstruction is generally lossy because of the 8-bit quantization. The displayed ciphertext images contain the
three encrypted imaginary components of $\mathcal{S}$.

Decomposition error uses the relative tensor residual in
\eqref{eq:tsvd-benchmark-error}. Recovery quality is measured before image or
video export, after clipping reconstructed RGB values to $[0,1]$.
For each frame, RGB errors are pooled into one MSE and
$\mathrm{PSNR}=10\log_{10}(1/\mathrm{MSE})$; channel SSIM values are averaged
using MATLAB \texttt{ssim} with \texttt{DynamicRange=1}
\cite{PSNRSSIM,Wang2004}. Video scores average all frames.

\subsection{Reconstruction Quality and Ciphertext Statistics}
Table~\ref{tab:app2-quality} reports the decomposition accuracy and reconstruction quality for the two test videos using Algorithm~2,
with PSNR and SSIM averaged over all frames of each video. The relative decomposition errors are below $9\times10^{-15}$, indicating high numerical accuracy. The reconstructed kite-surf and boat videos attain PSNR values of 62.3858 and 63.8569~dB, respectively, with corresponding SSIM values of 0.999738 and 0.999810. These values
indicate small reconstruction errors in both test videos. The quantized integer arrays are recovered exactly in both runs; component
quantization accounts for the remaining reconstruction discrepancy. Figure~\ref{fig:app2-frames} shows that the reconstructed frames closely reproduce the original scene details.

\begin{table}[!htbp]
\centering\small
\caption{Decomposition accuracy and reconstruction quality.}
\label{tab:app2-quality}
\begin{tabular}{lrrrr}
\toprule
Video & Frames & $e_{\rm dec}$ & PSNR (dB) & SSIM\\
\midrule
kite-surf & 50 & $8.6968\times10^{-15}$ & 62.3858 & 0.999738 \\
boat & 75 & $5.0514\times10^{-15}$ & 63.8569 & 0.999810 \\
\bottomrule
\end{tabular}
\end{table}

\begin{figure}[!htbp]
\centering
\setlength{\tabcolsep}{1pt}
\renewcommand{\arraystretch}{0.9}
\begin{tabular}{@{}c*{6}{c}@{}}
&\multicolumn{3}{c}{\textit{kite-surf}}&\multicolumn{3}{c}{\textit{boat}}\\[2pt]
&\footnotesize Frame 1&\footnotesize Frame 25&\footnotesize Frame 50&\footnotesize Frame 1&\footnotesize Frame 25&\footnotesize Frame 50\\[2pt]
\rotatebox{90}{\scriptsize Original} & \includegraphics[width=.150\linewidth]{application-figures/kite-surf_plain_1.png} & \includegraphics[width=.150\linewidth]{application-figures/kite-surf_plain_25.png} & \includegraphics[width=.150\linewidth]{application-figures/kite-surf_plain_50.png} & \includegraphics[width=.150\linewidth]{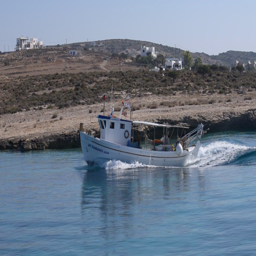} & \includegraphics[width=.150\linewidth]{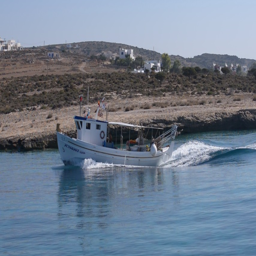} & \includegraphics[width=.150\linewidth]{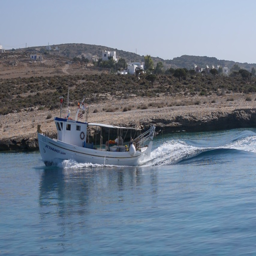}\\[-1pt]
\rotatebox{90}{\scriptsize Encrypted} & \includegraphics[width=.150\linewidth]{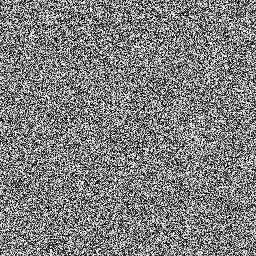} & \includegraphics[width=.150\linewidth]{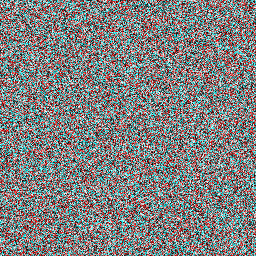} & \includegraphics[width=.150\linewidth]{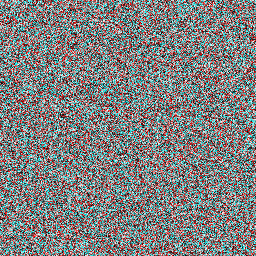} & \includegraphics[width=.150\linewidth]{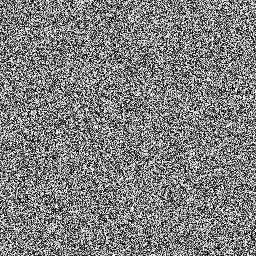} & \includegraphics[width=.150\linewidth]{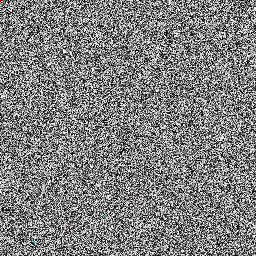} & \includegraphics[width=.150\linewidth]{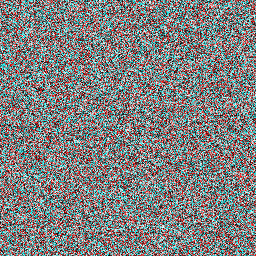}\\[-1pt]
\rotatebox{90}{\scriptsize Recovered} & \includegraphics[width=.150\linewidth]{application-figures/kite-surf_decrypted_1.png} & \includegraphics[width=.150\linewidth]{application-figures/kite-surf_decrypted_25.png} & \includegraphics[width=.150\linewidth]{application-figures/kite-surf_decrypted_50.png} & \includegraphics[width=.150\linewidth]{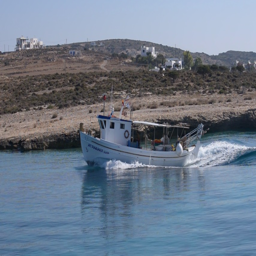} & \includegraphics[width=.150\linewidth]{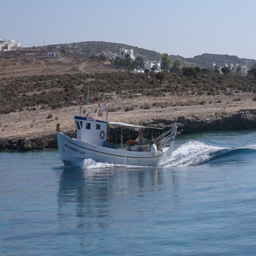} & \includegraphics[width=.150\linewidth]{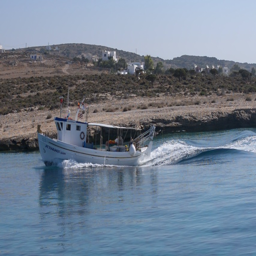}\\[-1pt]
\end{tabular}
\caption{Original videos (top), encrypted components of $\mathcal S$ (middle),
and recovered videos (bottom).}
\label{fig:app2-frames}
\end{figure}

Histograms and adjacent-pixel correlations are evaluated at frame~25.
Figure~\ref{fig:app2-histograms} compares the original RGB intensity
distributions with those of the three encrypted components of
$\mathcal S$ for both videos. For each original frame, the R, G, and B
histograms are overlaid, whereas the encrypted components are displayed
separately. The original distributions reflect the image content, while
the encrypted distributions are approximately uniform. For the adjacent-pixel correlation analysis, $3000$ horizontal
R-channel pixel pairs are sampled at identical spatial locations in the
original and encrypted images. Figure~\ref{fig:app2-statistical-panels}
shows that the encrypted pixel pairs are widely dispersed. Their
correlation coefficients are $-0.015782$ and $-0.015349$ for
\textit{kite-surf} and \textit{boat}, respectively, compared with
$0.982162$ and $0.975509$ for the corresponding original frames.
Thus, the strong horizontal correlation present in the original images
is substantially reduced after encryption.

Differential sensitivity is evaluated by perturbing a single plaintext
value, $R_{111}\leftarrow \operatorname{mod}(R_{111}+0.01,1),$ and repeating the decomposition and encryption procedure. The number of pixels change rate (NPCR) measures the percentage of changed ciphertext
values, whereas the unified average changing intensity (UACI) measures
the mean absolute ciphertext difference, normalized by $255$ and
expressed as a percentage~\cite{NCPRUACI}. Both quantities are averaged over
all frames and all three channels. As reported in
Table~\ref{tab:app2-statistics}, the NPCR and UACI values are close to
the reference values $99.6094\%$ and $33.4635\%$, respectively, for
independent uniformly distributed 8-bit data. The results show that the
single-pixel perturbation propagates throughout the ciphertext under the
tested setting.

\begin{table}[!htbp]
\centering\small
\caption{Ciphertext statistics.}
\label{tab:app2-statistics}
\begin{tabular}{lrrrr}
\toprule
Video & $r_{\rm plain}$ & $r_{\rm cipher}$ & NPCR (\%) & UACI (\%)\\
\midrule
kite-surf & 0.982162 & -0.015782 & 99.6119 & 33.4746 \\
boat & 0.975509 & -0.015349 & 99.6110 & 33.4517 \\
\bottomrule
\end{tabular}
\end{table}

\begin{figure}[H]
\centering
\setlength{\tabcolsep}{3pt}
\begin{tabular}{@{}cc@{}}
\textit{kite-surf}&\textit{boat}\\[2pt]
\includegraphics[width=.455\linewidth]{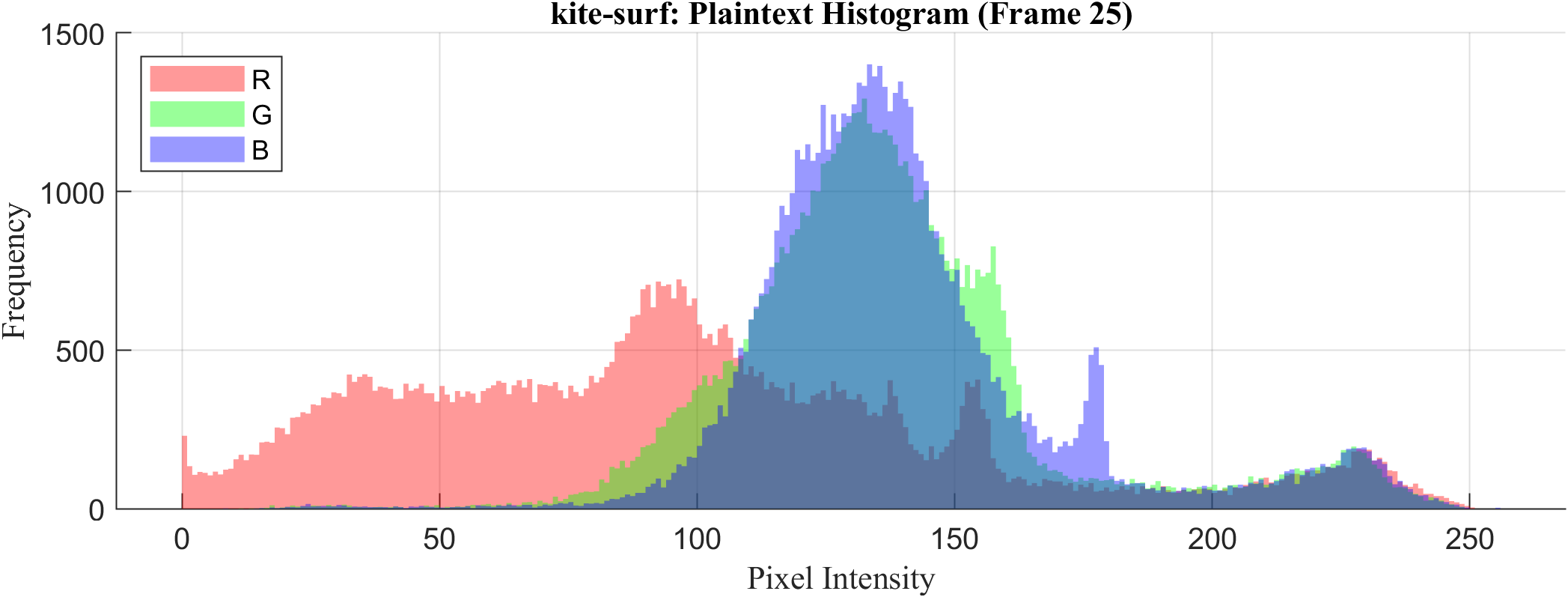} & \includegraphics[width=.455\linewidth]{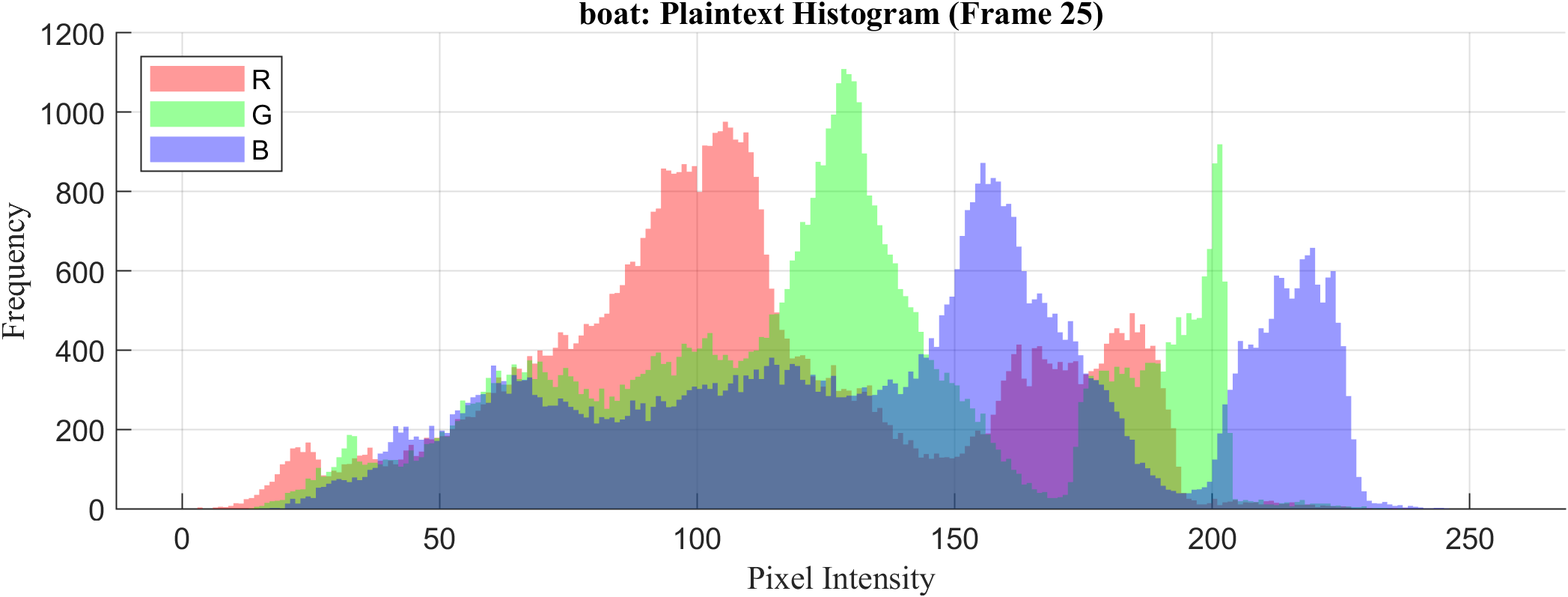}\\[2pt]
\includegraphics[width=.455\linewidth]{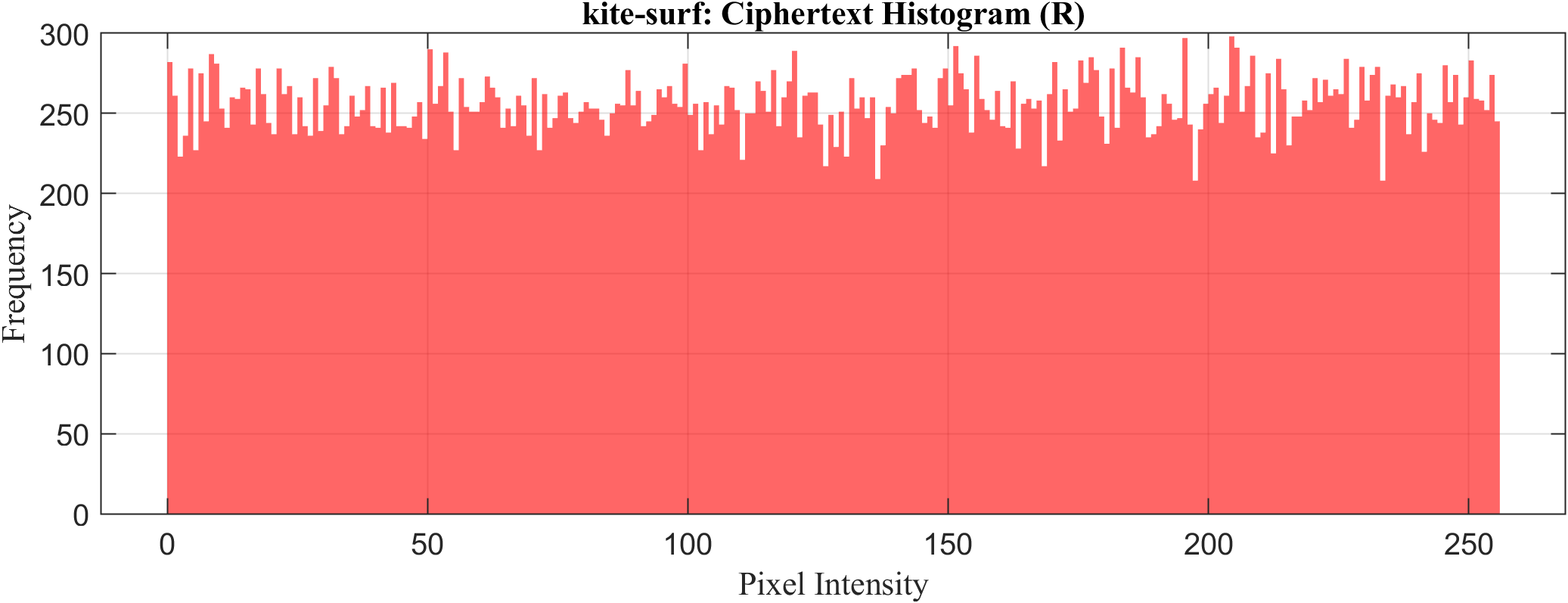} & \includegraphics[width=.455\linewidth]{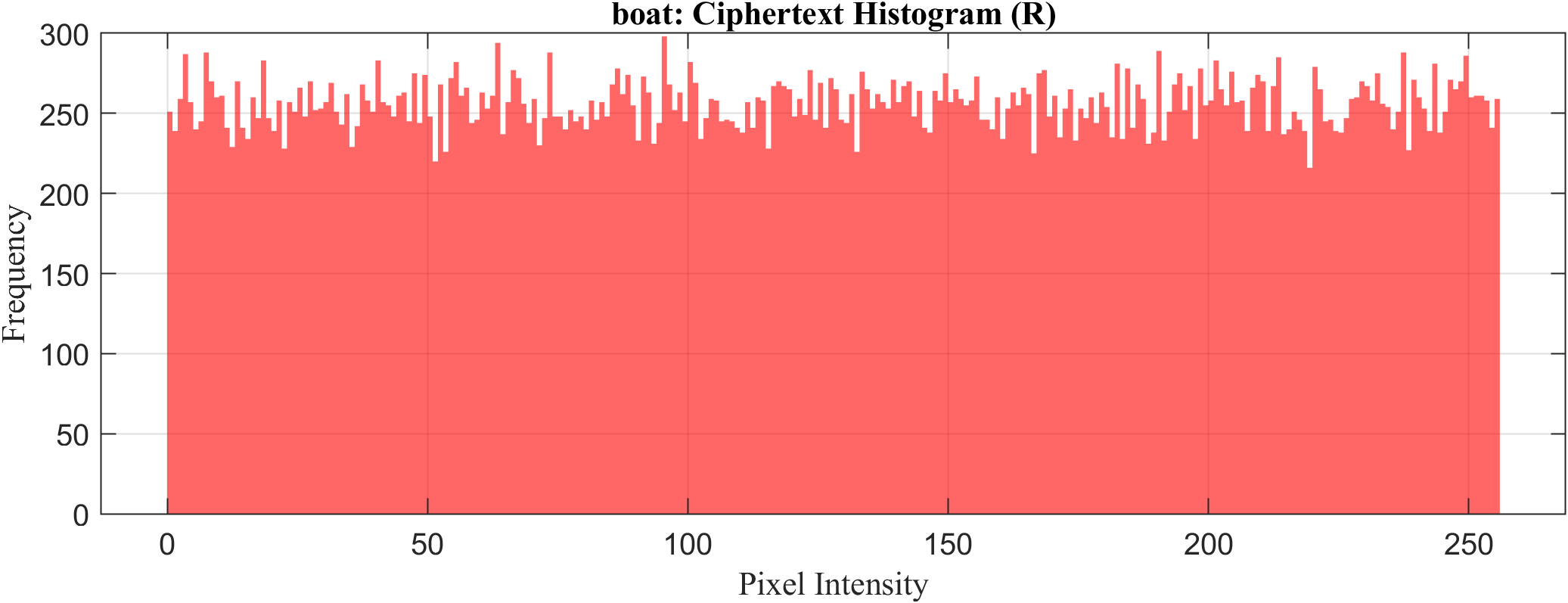}\\[2pt]
\includegraphics[width=.455\linewidth]{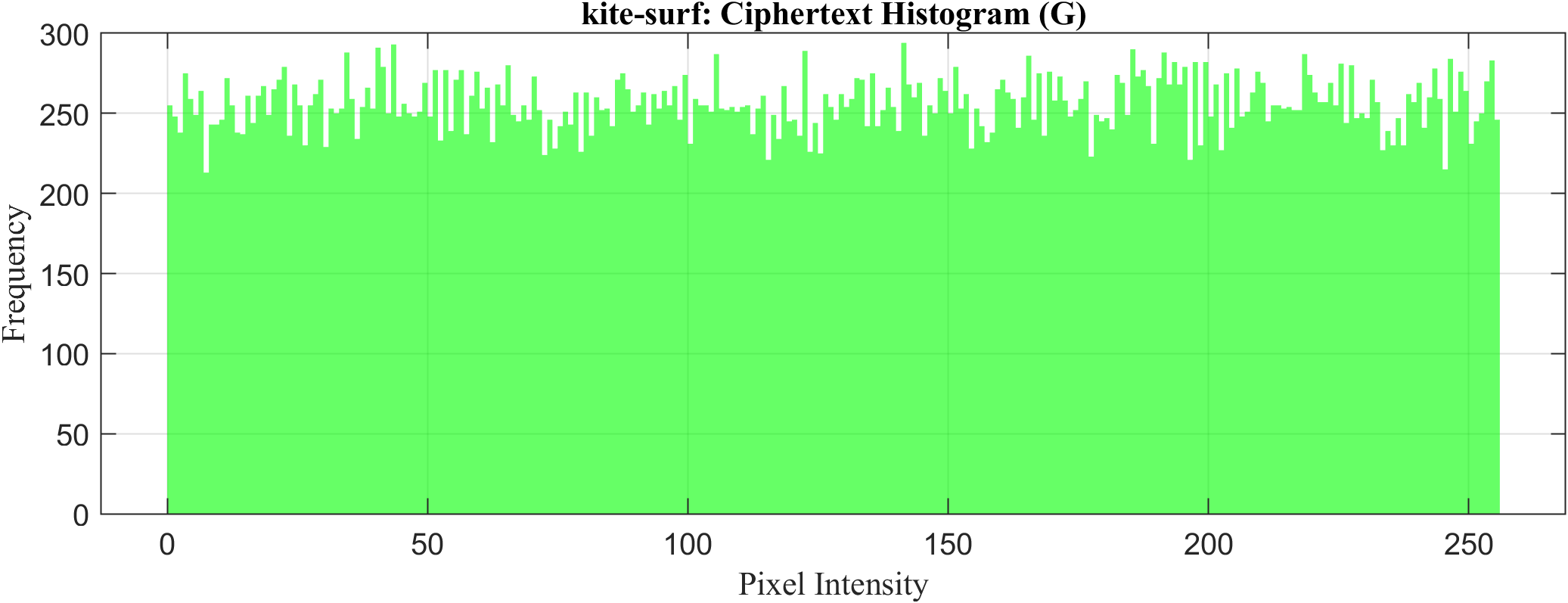} & \includegraphics[width=.455\linewidth]{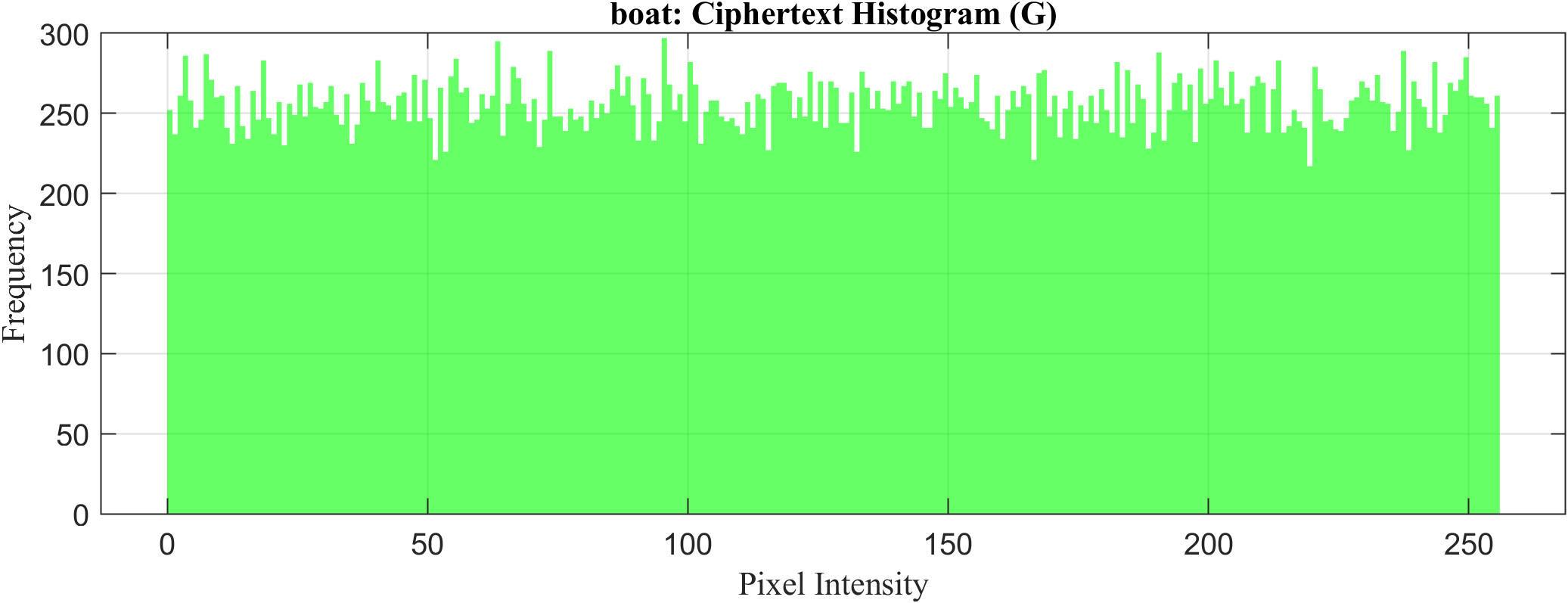}\\[2pt]
\includegraphics[width=.455\linewidth]{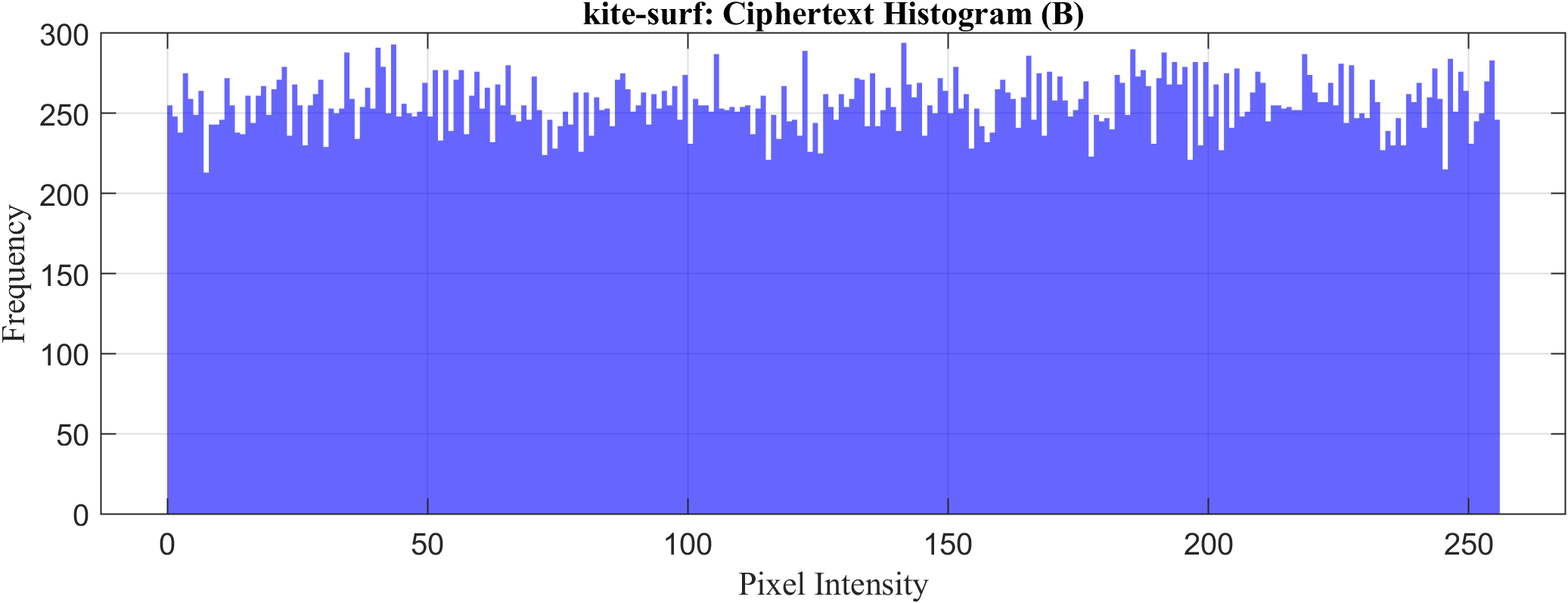} & \includegraphics[width=.455\linewidth]{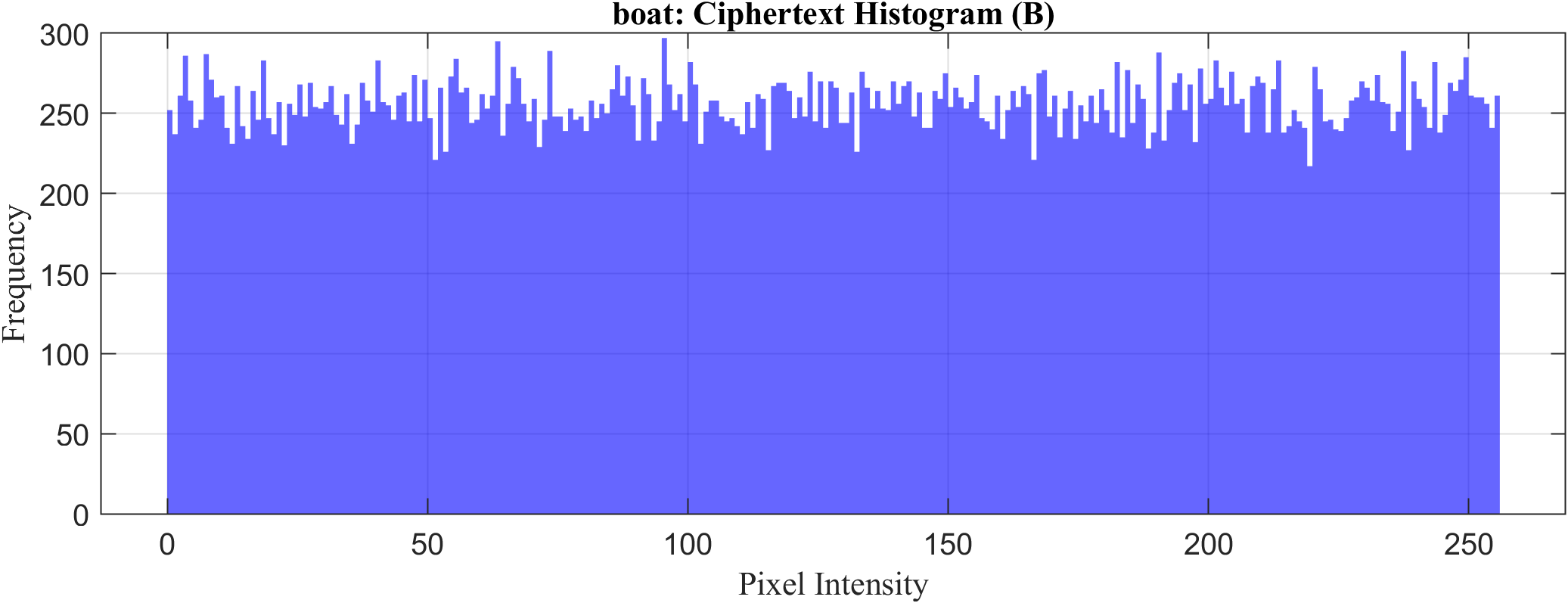}\\[2pt]
\end{tabular}
\caption{Histograms of the original and encrypted data at frame~25.}
\label{fig:app2-histograms}
\end{figure}

\begin{figure}[H]
\centering
\setlength{\tabcolsep}{2pt}
\begin{tabular}{@{}cccc@{}}
\multicolumn{2}{c}{\textit{kite-surf}}&\multicolumn{2}{c}{\textit{boat}}\\[2pt]
Original&Encrypted&Original&Encrypted\\[2pt]
\includegraphics[width=.238\linewidth]{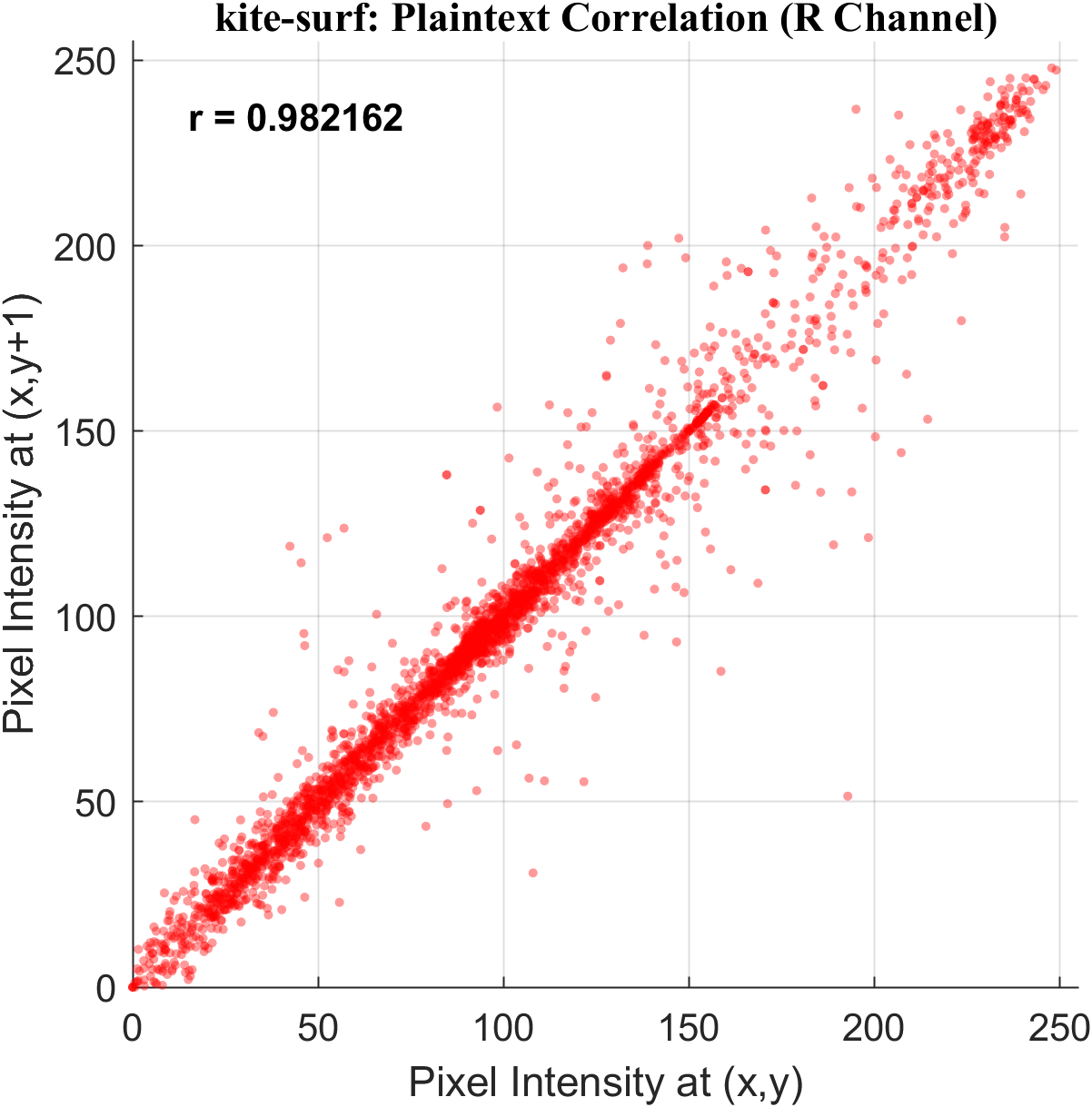} & \includegraphics[width=.238\linewidth]{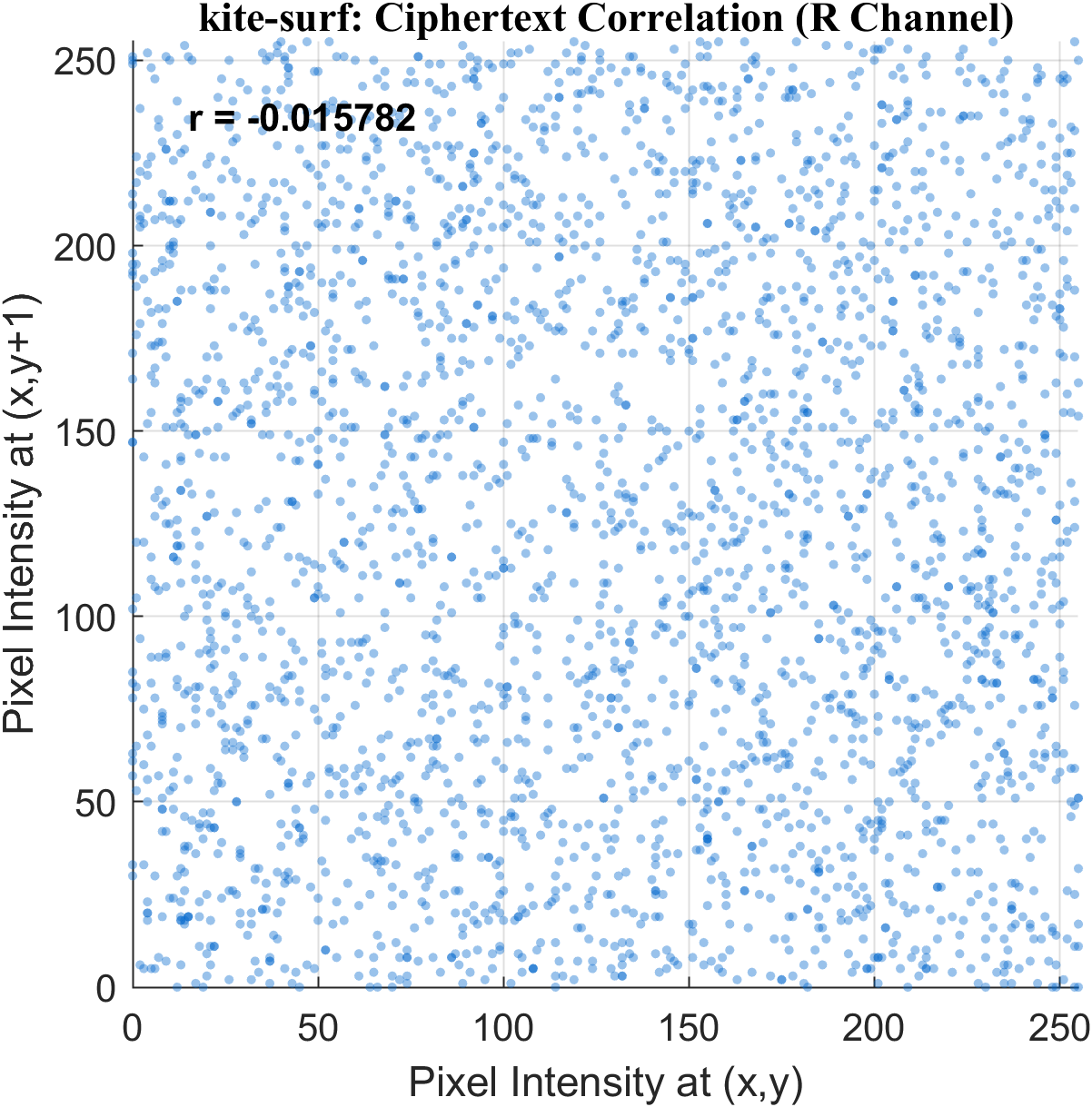} & \includegraphics[width=.238\linewidth]{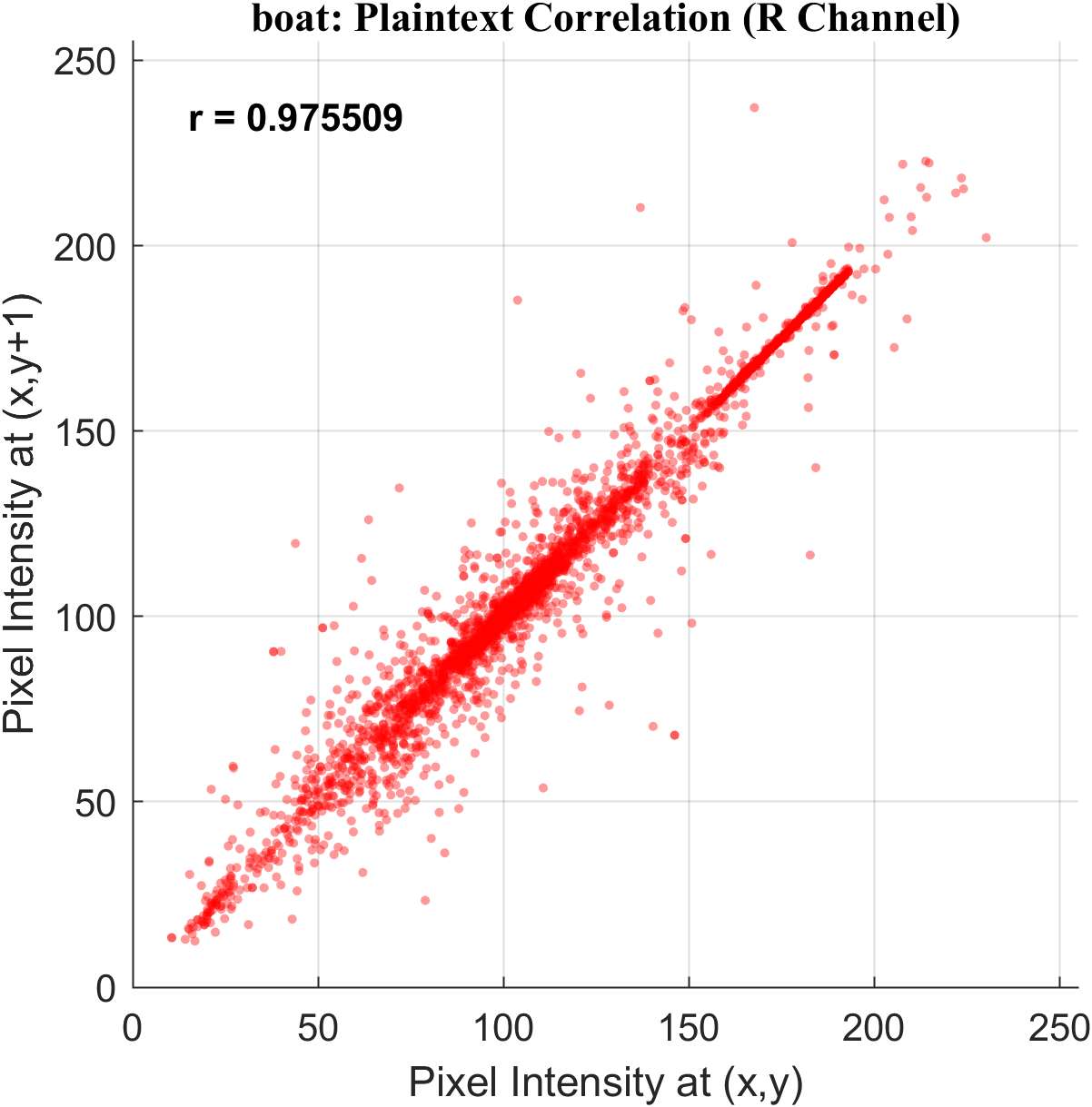} & \includegraphics[width=.238\linewidth]{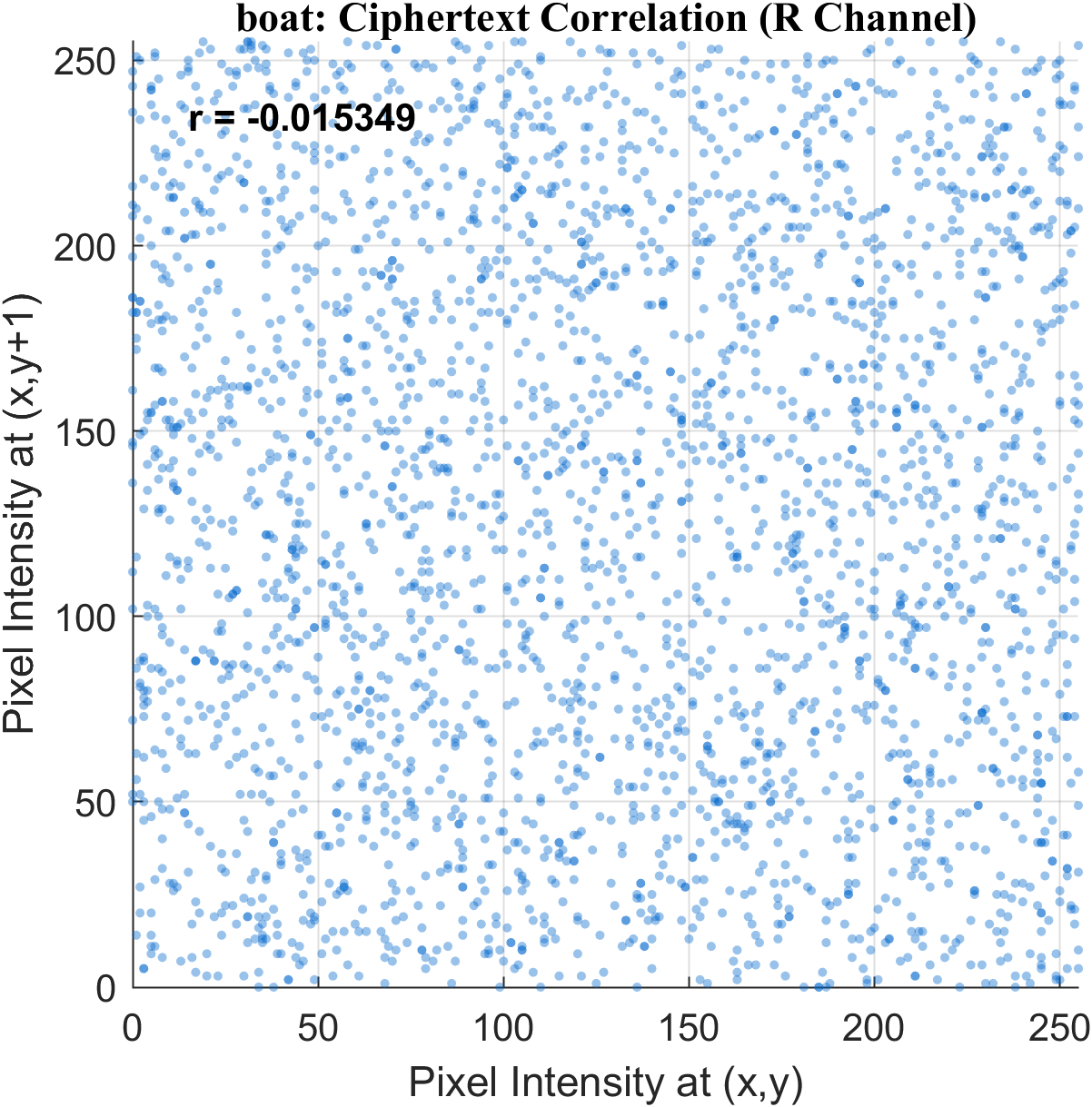}\\
\end{tabular}
\caption{Horizontal adjacent-pixel correlations at frame~25 (R Channel).}
\label{fig:app2-statistical-panels}
\end{figure}

\subsection{Stability under Input Perturbations}
We assess stability under input perturbations by separately adding
salt-and-pepper noise with density $0.01$, Gaussian noise with zero mean
and variance $0.01$, and a centered $12\times12$ zero occlusion to every
input frame before decomposition. Each perturbed video is then processed
through the same decomposition, encryption, decryption, and reconstruction
procedure.

Both the perturbed input and the recovered output are compared with the
corresponding clean video. Let $P_{\mathrm{in}}$ and $P_{\mathrm{out}}$
denote the corresponding PSNR values, and let $S_{\mathrm{in}}$ and
$S_{\mathrm{out}}$ denote the SSIM values. We further define
\[
\Delta P=P_{\mathrm{out}}-P_{\mathrm{in}}.
\]
As shown in Table~\ref{tab:app2-perturbation}, across all six cases,
$|\Delta P|$ is below $0.006$~dB and
$|S_{\mathrm{out}}-S_{\mathrm{in}}|$ is below
$5.1\times10^{-4}$. Thus, the processing pipeline introduces little
additional distortion beyond that already present in the perturbed
inputs. These experiments evaluate stability under pre-encryption input
perturbations rather than robustness to ciphertext corruption.

\begin{table}[!htbp]
\centering\small
\caption{Input and output quality under perturbations.}
\label{tab:app2-perturbation}
\setlength{\tabcolsep}{4pt}
\begin{tabular}{llrrrrr}
\toprule
Video & Perturbation & $P_{\rm in}$ & $P_{\rm out}$ & $\Delta P$ & $S_{\rm in}$ & $S_{\rm out}$\\
\midrule
kite-surf & Salt-and-pepper & 25.5216 & 25.5271 & $+0.005502$ & 0.767155 & 0.766830 \\
 & Gaussian & 20.1238 & 20.1245 & $+0.000723$ & 0.261248 & 0.261271 \\
 & Occlusion & 31.0040 & 31.0042 & $+0.000227$ & 0.995296 & 0.994959 \\
\addlinespace[3pt]
boat & Salt-and-pepper & 25.4392 & 25.4450 & $+0.005843$ & 0.782845 & 0.782335 \\
 & Gaussian & 20.0837 & 20.0828 & $-0.000875$ & 0.303207 & 0.303164 \\
 & Occlusion & 36.2688 & 36.2668 & $-0.001973$ & 0.996884 & 0.996688 \\
\bottomrule
\end{tabular}
\end{table}
\FloatBarrier

\section*{Funding}
\begin{itemize}
\item The first author was supported by the National Natural Science Foundation of China (NSFC) (No. 11901079), and China Postdoctoral Science Foundation (No. 2021M700751), and the Scientific and Technological Research Program Foundation of Jilin Province (No. JJKH20190690KJ; No. JJKH20220091KJ; No. JJKH20250851KJ).
\end{itemize}

\section*{Conflict of Interest}
The authors declare that they have no potential conflict of interest.

\section*{Data Availability}
Data sharing is not applicable to this article, as no datasets were generated or analyzed during the current study.

\section*{Acknowledgments}

The authors are grateful to Michael K. Ng of Hong Kong Baptist University for his valuable suggestions, which have helped improve the quality of this paper.

\bibliographystyle{abbrv}

\end{document}